\documentclass[10pt,psamsfonts]{article}

\usepackage{amsmath, amsfonts, amsthm, amssymb, enumerate, url}
\usepackage[all,arc]{xy}
\usepackage{enumerate}
\usepackage{mathrsfs}
\usepackage{mathtools}
\usepackage[utf8]{inputenc}
\usepackage{tikz-cd}

\begingroup
    \makeatletter
    \@for\theoremstyle:=definition,remark,plain\do{%
        \expandafter\g@addto@macro\csname th@\theoremstyle\endcsname{%
            \addtolength\thm@preskip\parskip
            }%
        }
 \endgroup
 
   \usepackage[titletoc,toc,title]{appendix}

   \usepackage[colorlinks = true,
            linkcolor = blue,
            urlcolor  = blue,
            citecolor = blue,
            anchorcolor = blue]{hyperref}
            
   \usepackage[parfill]{parskip}
   \usepackage{anysize}
   \marginsize{1.2in}{1.2in}{1in}{1in}
   \usepackage{thmtools}

\declaretheorem[name=Theorem,numberwithin=section]{thm}
\declaretheorem[name=Proposition,numberlike=thm]{prop}
\declaretheorem[name=Lemma,numberlike=thm]{lemma}
\declaretheorem[name=Corollary,numberlike=thm]{cor}

\declaretheorem[name=Definition,style=definition,qed=$\blacktriangle$,numberlike=thm]{defn}
\declaretheorem[name=Example,style=definition,qed=$\blacktriangle$,numberlike=thm]{ex}
\declaretheorem[name=Remark,style=definition,qed=$\blacktriangle$,numberlike=thm]{rmk}

\newcommand{\bu}{\bullet}
\newcommand{\dd}{\mathrm{d}}
\newcommand{\cc}{\mathrm{c}}

\newcommand{\cop}[1]{{#1}{}^{\cc}}
\newcommand{\dc}{\cop{\dd}}

\newcommand{\ddc}{\dd \dc}

\newcommand{\imag}{\mathrm{Im}}
\newcommand{\G}{\mathrm{G}_2}

\newcommand{\GL}[1]{\mathrm{GL}(#1)}

\newcommand{\U}[1]{\mathrm{U}(#1)}

\newcommand{\R}{\mathbb R}
\newcommand{\C}{\mathbb C}

\newcommand{\Z}{\mathbb Z}

\newcommand{\del}{\partial}
\newcommand{\delbar}{\overline{\partial}}

\newcommand{\DR}{\mathrm{dR}}
\newcommand{\BC}{\mathrm{BC}}

\newcommand{\im}{\operatorname{im}}
\newcommand{\coker}{\operatorname{coker}}

\newcommand{\End}{\operatorname{End}}

\makeatletter
\let\c@equation\c@thm
\makeatother
\numberwithin{equation}{section}

\begin{document}

\title{Cohomologies on almost complex manifolds \\ and the $\partial \bar{\partial}$-lemma}

\author{Ki Fung Chan \\ {\it Chinese University of Hong Kong} \\\tt{kifung@link.cuhk.edu.hk} \and Spiro Karigiannis \\ {\it Department of Pure Mathematics, University of Waterloo} \\ \tt{karigiannis@uwaterloo.ca} \and Chi Cheuk Tsang \\ {\it Chinese University of Hong Kong} \\ \tt{1155062770@link.cuhk.edu.hk} }

\maketitle

\begin{abstract}
We study cohomologies on an almost complex manifold $(M, J)$, defined using the Nijenhuis-Lie derivations $\mathcal{L}_J$ and $\mathcal{L}_N$ induced from the almost complex structure $J$ and its Nijenhuis tensor $N$, regarded as vector-valued forms on $M$. 

We show how one of these, the $N$-cohomology $H^{\bu}_N (M)$, can be used to distinguish non-isomorphic non-integrable almost complex structures on $M$. Another one, the $J$-cohomology $H^{\bu}_J (M)$, is familiar in the integrable case but we extend its definition and applicability to the case of non-integrable almost complex structures. The $J$-cohomology encodes whether a complex manifold satisfies the $\partial \bar{\partial}$-lemma, and more generally in the non-integrable case the $J$-cohomology encodes whether $(M, J)$ satisfies the $\dd \mathcal{L}_J$-lemma, which we introduce and motivate in this paper. We discuss several explicit examples in detail, including a non-integrable example.

We also show that $H^k_J$ is finite-dimensional for compact integrable $(M, J)$, and use spectral sequences to establish partial results on the finite-dimensionality of $H^k_J$ in the compact non-integrable case.

\smallskip

\noindent \textcolor{red}{Changes made after publication: Theorem 3.23 was originally misstated. Also, we have made minor changes to notation in Theorem 3.32. The authors had originally unwisely chosen (without stating it explicitly) to define $H^{k-1}( \im \mathcal{L}_J) = \big( (\ker \dd )^k \cap (\im \mathcal{L}_J)^k \big) / (\im \dd)^k$. This was done to make the map of complexes $\Omega^{\bu} \to (\im \mathcal{L}_J)^{\bu}$ degree preserving. However this choice was unnecessary and caused unintended confusion. We have now sensibly defined $H^k( \im \mathcal{L}_J) = \big( (\ker \dd )^k \cap (\im \mathcal{L}_J)^k \big) / (\im \dd)^k$. All corrections appear in red. The authors thank Scott Wilson for alerting them to these errors.}
\end{abstract}

\tableofcontents

\section{Introduction} \label{sec:intro}

Given an almost complex manifold $(M, J)$, we obtain two Nijenhuis-Lie derivations $\mathcal{L}_J$ and $\mathcal{L}_N$ induced from the almost complex structure $J$ and its Nijenhuis tensor $N$, regarded as vector-valued forms. The motivation for the present paper is the study of the derivations $\mathcal{L}_J$ and $\mathcal{L}_N$. We define cohomologies on $(M, J)$ using these operators, and consider applications of these cohomologies. One of these, the $J$-cohomology $H^{\bullet}_J (M)$ is already familiar in complex (integrable $J$) geometry, but we extend its definition and applicability to the non-integrable case. We emphasize being able to explicitly compute these cohomologies with several particular examples.

The study in the present paper is a natural extension of the general study of the use of derivations to understand $\GL{m, \C}$-structures and $\U{m}$-structures initiated in~\cite{dKS} by the second author together with X. de la Ossa and E.E. Svanes. Although what we call the ``$J$-cohomology'' in the present paper has been considered before in the case of complex (integrable $J$) geometry, our approach is more general and emphasizes a different point of view, namely that of Nijenhuis-Lie derivations. Our main new result is an extension to the non-integrable case of both the definition of the $J$-cohomology and the establishment of its corresponding relation (Theorem~\ref{thm:main2}) to a generalization of the $\partial \bar \partial$-lemma, which we call the $\dd \mathcal{L}_J$-lemma. We discuss the relations to previously known results in Section~\ref{sec:del-delbar}.

The paper is organized as follows. In Section~\ref{sec:first} we review the notion of derivations on the space of differential forms, and specialize to the case of an almost complex manifold $(M, J)$. We define three cohomologies on $(M, J)$ and establish some of their basic properties. We also present an explicit example to demonstrate that the ``$N$-cohomology'' can be used to distinguish non-isomorphic non-integrable almost complex structures. In Section~\ref{sec:Jcohom} we focus on the ``$J$-cohomology'' and many of its properties, including: its finite-dimensionality in certain cases; its relation to the $\partial \bar \partial$-lemma in the integrable case, and more generally to the $\dd \mathcal{L}_J$-lemma in the non-integrable case; and several explicit examples. In Section~\ref{sec:future} we discuss some natural questions for future study.

{\bf Remark.} What we call the $J$-cohomology and denote by $H^{\bu}_J$, is not related to the groups $H^{(p,q)}_J$ defined in~\cite[Section 2.1]{Angella}, which play no role in the present paper. Our notation for $J$-cohomology should therefore not cause any confusion.

{\bf Notation and conventions.} We use $H^k_{\DR} (M)$ to denote the $k^{\text{th}}$ de Rham cohomology (with real coefficients) of the manifold $M$, and write just $H^k_{\DR}$ when $M$ is understood. We make similar notational abbreviations for the other cohomologies introduced in this paper.

We use $C^{\bu}$ to denote a $\Z$-graded complex of real vector spaces. A degree $k$ map $P$ of the complex $C^{\bu}$ maps $C^i$ into $C^{i+k}$, and we write
\begin{align*}
(\ker P)^i & = \ker (P : C^i \to C^{i+k}), \\
(\im P)^i & = \im (P : C^{i-k} \to C^i).
\end{align*}

{\bf Acknowledgments.} These results were obtained in 2017 as part of the collaboration between the COSINE program organized by the Chinese University of Hong Kong and the URA program organized by the University of Waterloo. The authors thank both universities for this opportunity. Part of the writing was done while the second author held a Fields Research Fellowship at the Fields Institute. The second author thanks the Fields Institute for their hospitality.\\

\section{Derivations and cohomologies} \label{sec:first}

In this section we review the notion of \emph{derivations} on the space of forms on a smooth manifold, and define the canonical derivations associated to an almost complex manifold $(M, J)$. These derivations will be used to define cohomologies in Section~\ref{sec:cohom}.

\subsection{Derivations on almost complex manifolds} \label{sec:deriv}

For a smooth manifold $M^n$, let $\Omega^k(M) = \Gamma(\Lambda^k (T^*M))$ be the space of $k$-forms on $M$ and let $\Omega^k(M, TM) = \Gamma(\Lambda^k (T^* M) \otimes TM)$ be the space of vector-valued $k$-forms on $M$. Given an element $K \in \Omega^k (M, TM)$, it induces two derivations on the graded algebra $\Omega^{\bu}(M) = \oplus_{j=0}^n \Omega^j (M)$ of all forms. They are the \emph{algebraic derivation} $\iota_K$, which is of degree $k-1$, and the \emph{Nijenhuis-Lie derivation} $\mathcal{L}_K$, which is of degree $k$. These derivations are defined as follows. Let $\{ e_1, \ldots, e_n \}$ be a local frame for $TM$ over an open subset $U$ of $M$ with dual local coframe $\{ e^1, \ldots, e^n \}$ for $T^* M$. Then $K = K^j e_j$ where each $K^j$ is a $k$-form on $U$.

The map $\iota_K : \Omega^i (M) \to \Omega^{i+k-1} (M)$ is defined as
\begin{equation} \label{eq:alg-derivation}
\iota_K \alpha = K^j \wedge (\iota_{e_j} \alpha),
\end{equation}
where $\iota_{e_j}$ is the usual interior product with a vector field. It is easy to see $\iota_K$ is well-defined and is a derivation on $\Omega^{\bu} (M)$. In particular, $\iota_K$ vanishes on functions, so $\iota_K (h \alpha) = h (\iota_K \alpha)$ for any $h \in \Omega^0 (M)$ and $\alpha \in \Omega^{i} (M)$, which is why $\iota_K$ is called algebraic. Moreover, if $\alpha \in \Omega^1 (M)$, one can show that
\begin{equation} \label{eq:alg-derivation1}
(\iota_K \alpha) (X_1, \ldots, X_k) = \alpha( K(X_1, \ldots, X_k) ).
\end{equation}

The map $\mathcal{L}_K : \Omega^i (M) \to \Omega^{i+k} (M)$ is defined as
\begin{equation} \label{eq:Lie-derivation}
\mathcal{L}_K \alpha =  \iota_K (\dd \alpha) - (-1)^{k-1} \dd (\iota_K \alpha) = [ \iota_K, \dd ] \alpha.
\end{equation}
That is, $\mathcal{L}_K$ is the \emph{graded commutator} of $\iota_K$ and the exterior derivative $\dd$. In fact, $\dd$ itself is a Nijenhuis--Lie derivation, as $\dd = \mathcal{L}_I$, where $I = e^j e_j$ is the identity operator in $\Omega^1 (M, T M) = \Gamma(\End TM)$. Moreover, the linear map $K \mapsto \mathcal{L}_K$ is an injective map from $\Omega^k (M, TM)$ into the linear space of degree $k$ derivations of $\Omega^{\bu} (M)$. The graded Jacobi identity on the space of graded linear operators on $\Omega^{\bu} (M)$ and the fact that $\dd^2 = 0$ implies that
\begin{equation} \label{eq:commdL}
[ \dd, \mathcal{L}_K ] = \dd \mathcal{L}_K - (-1)^k \mathcal{L}_K \dd = 0.
\end{equation}

Also important for us will be the \emph{Fr\"olicher-Nijenhuis bracket}
\begin{equation*}
\{ \cdot, \cdot \} : \Omega^k (M, TM) \times \Omega^l (M, TM) \to \Omega^{k+l} (M, TM)
\end{equation*}
on vector-valued forms which can be defined by
\begin{equation} \label{eq:FNdefn}
\mathcal{L}_{\{ K, L \}} = [ \mathcal{L}_K, \mathcal{L}_L ] = \mathcal{L}_K \mathcal{L}_L - (-1)^{kl} \mathcal{L}_L \mathcal{L}_K
\end{equation}
for any $K \in \Omega^k (M, TM)$ and $L \in \Omega^l (M, TM)$. That is, the Nijenhuis-Lie derivative in the $\{ K, L \}$ direction is the graded commutator of $\mathcal{L}_K$ with $\mathcal{L}_L$. In particular, it follows from this definition that
\begin{equation} \label{eq:FNsymm}
\{ L, K \} = - (-1)^{kl} \{ K, L \},
\end{equation}
and that
\begin{equation} \label{eq:FNodd}
(\mathcal{L}_K)^2 = \tfrac{1}{2} \mathcal{L}_{\{K, K\}} \text{ when $k$ is odd}.
\end{equation}
Hence we deduce that
\begin{equation} \label{eq:FNeq}
\begin{aligned}
\{ K, K \} = 0 & \text{ always when $k$ is even}, \\
\{ K, K \} = 0 & \text{ if and only if $(\mathcal{L}_K)^2 = 0$ when $k$ is odd}.
\end{aligned}
\end{equation}
A further consequence of the definition of the Fr\"olicher-Nijenhuis bracket and the graded Jacobi identity is that $\{ \cdot, \cdot \}$ also satisfies a graded Jacobi identity:
\begin{equation*}
\{ K, \{L, P \} \} = \{ \{ K, L\}, P \} + (-1)^{kl} \{ L, \{K, P \} \}.
\end{equation*}
From the above identity and~\eqref{eq:FNsymm} we find that
\begin{equation} \label{eq:FNJacobi}
\{ K, \{ K, K \} \} = 0 \text{ always when $k$ is odd}.
\end{equation}
A good reference for the Fr\"olicher-Nijenhuis bracket, the algebraic and Nijenhuis-Lie derivations, and their various properties and relations is~\cite{KMS}. 

From now on let $n = \dim M = 2m$ and let $(M^{2m}, J)$ be an \emph{almost complex manifold}. This means that $J \in \Omega^1 (M, TM) = \Gamma(\End TM)$ such that $J^2 = - I$. Since $J$ is a vector-valued $1$-form we have induced derivations $\iota_J$ and $\mathcal{L}_J$ of degree $0$ and $1$, respectively.

We note for future reference from~\eqref{eq:Lie-derivation} that
\begin{equation} \label{eq:LJspecial}
\text{$\mathcal{L}_J = \iota_J \dd$ on $\Omega^0$ \quad and \quad $\mathcal{L}_J = - \dd \iota_J$ on $\Omega^n$.}
\end{equation}

Associated to $J$ is the \emph{Nijenhuis tensor} $N \in \Omega^2 (M, TM)$. One way to define $N$ is in terms of the Fr\"olicher-Nijenhuis bracket by $N = -\tfrac{1}{2} \{ J, J \}$. The factor of $-\tfrac{1}{2}$ is not important and is inserted here only to match the most common convention for $N$, which is
\begin{equation} \label{eq:Nclassical}
N(X, Y) = [X, Y] + J [JX, Y] + J[X, JY] - [JX, JY].
\end{equation}
Note from~\eqref{eq:FNodd} and~\eqref{eq:FNeq} that
\begin{equation} \label{eq:LJsquared}
(\mathcal{L}_J)^2 = - \mathcal{L}_N.
\end{equation}
In particular we deduce that
\begin{equation} \label{eq:Nzero}
N = 0 \text{ if and only if } (\mathcal{L}_J)^2 = 0.
\end{equation}
If the Nijenhuis tensor vanishes we say that $J$ is \emph{integrable}, and by the Newander-Nirenberg Theorem this happens if and only if $M$ admits an atlas of \emph{holomorphic charts} making $M$ into a \emph{complex manifold} such that $J_p$ corresponds to multiplication by $i$ on each $T_p M \cong \C^{m}$. From $N$ we get induced derivations $\iota_N$ and $\mathcal{L}_N$ of degree $1$ and $2$, respectively.

\subsection{Cohomologies defined using $\dd$, $\mathcal{L}_J$, and $\mathcal{L}_N$} \label{sec:cohom}

In this section we define the new cohomologies and discuss some of their basic properties.

Recall from Section~\ref{sec:deriv} that on an almost complex manifold $(M, J)$ with Nijenhuis tensor $N$, we have three Nijenhuis-Lie derivations: $\dd$, $\mathcal{L}_J$, and $\mathcal{L}_N$. Using~\eqref{eq:LJsquared} and~\eqref{eq:Nzero} we can summarize their properties so far as follows:
{\renewcommand{\arraystretch}{1.2}
\begin{center}
\begin{tabular}{|c|c|c|}
\hline
Derivation & Degree & Square \\ \hline
$\dd$ & 1 & $\dd^2 = 0$ always \\
$\mathcal{L}_J$ & 1 & $(\mathcal{L}_J)^2 = - \mathcal{L}_N$, which vanishes iff $N = 0$ \\
$\mathcal{L}_N$ & 2 & $(\mathcal{L}_N)^2 \neq 0$ in general \\
\hline
\end{tabular}
\end{center}
}
We seek to define cohomologies on $(M, J)$ using these three operators. Thus we look for a subspace of the space $\Omega^{\bu} (M)$ of forms on which one of these operators squares to zero. Of course, such a subspace would need to be invariant under the operator. To this end, we observe the following.
\begin{lemma} \label{lemma:commutators}
The following graded commutations relations hold:
\begin{equation*}
[ \dd, \mathcal{L}_J ] = 0, \qquad [ \dd, \mathcal{L}_N ] = 0, \qquad [ \mathcal{L}_J , \mathcal{L}_N ] = 0.
\end{equation*}
Consequently, if $P, Q$ are any two of the operators $\dd, \mathcal{L}_J, \mathcal{L}_N$, then $P$ maps $\ker Q$ into itself.
\end{lemma}
\begin{proof}
This first two relations are~\eqref{eq:commdL}. For the third relation, using~\eqref{eq:FNdefn} and~\eqref{eq:FNJacobi} we compute
\begin{equation*}
[ \mathcal{L}_J, \mathcal{L}_N ] = [ \mathcal{L}_J, - \tfrac{1}{2} \mathcal{L}_{\{J,J\}} ] = - \tfrac{1}{2} \mathcal{L}_{\{J, \{J, J\} \}} = 0
\end{equation*}
as claimed.
\end{proof}
Now we ask when is $P^2 = 0$ on the subspace $\ker Q$.

\emph{Case One:} $P= \dd$. Then $P^2 = 0$ always, so in particular it continues to be true on the subspaces $\ker \mathcal{L}_J$ and $\ker \mathcal{L}_N$. (If we do not restrict to $\ker Q$ then we just obtain de Rham cohomology.)

\emph{Case Two:} $P = \mathcal{L}_J$. Since $(\mathcal{L}_J)^2 = - \mathcal{L}_N$, we find that $P^2$ indeed always vanishes on the subspace $\ker \mathcal{L}_N$. On the other hand, to ensure that $P^2 = 0$ on the subspace $\ker \dd$, we would need to know that $(\ker \dd) \subseteq (\ker \mathcal{L}_N)$, but this inclusion does not hold in general.

\emph{Case Three:} $P = \mathcal{L}_N$. In this case, $P = - (\mathcal{L}_J)^2$, so in fact $P = 0$ on $\ker \mathcal{L}_J$. Moreover, in general we have $P^2 = (\mathcal{L}_J)^4 \neq 0$ on $\ker \dd$.

\begin{rmk} \label{LNremark}
In fact, it is quite difficult to understand the operator $(\mathcal{L}_N)^2$, because unlike the case $(\mathcal{L}_J)^2 = - \mathcal{L}_N$, we \emph{do not} have that $(\mathcal{L}_N)^2$ is again a Nijenhuis-Lie derivation. Indeed, it need not be a derivation at all. Moreover, again unlike the case of $(\mathcal{L}_J)^2 = \tfrac{1}{2} \mathcal{L}_{\{J, J\}}$, the operator $(\mathcal{L}_N)^2$ is not related to the element $\{N, N\}$, which is always zero by~\eqref{eq:FNeq}.
\end{rmk}

The above discussion motivates the definition of the following three cohomologies on $(M, J)$.
\begin{defn} \label{defn:cohom}
Let $(M, J)$ be an almost complex manifold.
\begin{itemize}
\item The \emph{$J$-cohomology} of $(M, J)$ in degree $k$ is denoted by $H^k_J (M)$, and is defined to be the cohomology of the complex
\begin{equation*}
(\ker \mathcal{L}_J)^0 \overset{\dd}{\longrightarrow} (\ker \mathcal{L}_J)^1 \overset{\dd}{\longrightarrow} \cdots (\ker \mathcal{L}_J)^{2m - 1} \overset{\dd}{\longrightarrow} (\ker \mathcal{L}_J)^{2m}.
\end{equation*}
\emph{The $J$-cohomology is in general nontrivial even if $J$ is integrable.}
\item The \emph{$N$-cohomology} of $(M, J)$ in degree $k$ is denoted by $H^k_N (M)$, and is defined to be the cohomology of the complex
\begin{equation*}
(\ker \mathcal{L}_N)^0 \overset{\dd}{\longrightarrow} (\ker \mathcal{L}_N)^1 \overset{\dd}{\longrightarrow} \cdots (\ker \mathcal{L}_N)^{2m - 1} \overset{\dd}{\longrightarrow} (\ker \mathcal{L}_N)^{2m}.
\end{equation*}
\item The \emph{$J$-twisted $N$-cohomology} of $(M, J)$ in degree $k$ is denoted by $\widetilde{H}^k_N (M)$, and is defined to be the cohomology of the complex
\begin{equation*}
(\ker \mathcal{L}_N)^0 \overset{\mathcal{L}_J}{\longrightarrow} (\ker \mathcal{L}_N)^1 \overset{\mathcal{L}_J}{\longrightarrow} \cdots (\ker \mathcal{L}_N)^{2m - 1} \overset{\mathcal{L}_J}{\longrightarrow} (\ker \mathcal{L}_N)^{2m}.
\end{equation*}
\end{itemize}
\end{defn}
\begin{rmk} \label{rmk:cohomdefns}
We make the following observations about these three cohomologies.
\begin{enumerate}[(a)] \setlength\itemsep{-0.8mm}
\item Suppose the $J$ is integrable. That is, $N = 0$. Then $(\ker \mathcal{L}_N)^k = \Omega^k (M)$. In particular, in this case the $N$-cohomology $H^k_N (M)$ is just the de Rham cohomology.
\item Consider the operator $\dc = J^{-1} \dd J$, which is called the $J$-twist of $\dd$ in~\cite{dKS}. In~\cite[Section 3.3]{dKS} it is shown that
\begin{equation} \label{eq:dKS}
\dc = - \mathcal{L}_J - \iota_{J \cdot N},
\end{equation}
where $J \cdot N$ is a vector-valued $2$-form on $M$ defined by $(J \cdot N) (X, Y) = J ( N(JX, JY) )$. Because $\mathcal{L}_J$ is closely related to the $J$-twist $\dc$ of $\dd$, that is why we choose to call $\widetilde{H}^k_N (M)$ the $J$-twisted $N$-cohomology. If $N = 0$ then by (a) above $(\ker \mathcal{L}_N)^k = \Omega^k (M)$, and by~\eqref{eq:dKS} we have $\mathcal{L}_J = - \dc$. Thus when $N = 0$, the $J$-twisted $N$-cohomology is just the $\dc$ cohomology, which is clearly isomorphic to the de Rham cohomology.
\item By (a) and (b) above, both the $N$-cohomology and the $J$-twisted $N$-cohomology are only interesting if $J$ is \emph{not} integrable, that is when $N \neq 0$. By contrast, \emph{the $J$-cohomology $H^k_J (M)$ is in general nontrivial even if $J$ is integrable.}
\end{enumerate}
\end{rmk}
Next we investigate the functoriality of these three cohomologies.
\begin{defn} \label{defn:morphism}
A morphism between almost complex manifolds $(M,J)$ and $(M',J')$ is a smooth map $f: M \to M'$ such that $(f_*)_p \circ J_p = (J')_{f(p)} \circ (f_*)_p$ for all $p \in M$.
\end{defn}

\begin{prop} \label{prop:functoriality}
Let $f : (M, J) \to (M', J')$ be a morphism between almost complex manifolds. Then we have induced homomorphisms
\begin{align*}
f^* & : H^k_J (M') \to H^k_J (M), \\
f^* & : H^k_N (M') \to H^k_N (M), \\ 
f^* & : \widetilde{H}^k_N (M') \to \widetilde{H}^k_N (M),
\end{align*}
which obey the usual functoriality laws $I^* = I$ and $(fg)^* = g^* f^*$.
\end{prop}
\begin{proof}
For such a morphism $f$, we have
\begin{equation*}
f^* \dd = \dd f^*, \qquad f^*\mathcal{L}_{J'} = \mathcal{L}_J f^*, \qquad f^*\mathcal{L}_{N'} = \mathcal{L}_N f^*,
\end{equation*}
on $\Omega^{\bu} (M')$. (See~\cite[8.15]{KMS}). The result now follows.
\end{proof}

\begin{cor} \label{cor:isomorphic}
Isomorphic almost complex manifolds have isomorphic cohomologies $H^k_J$, $H^k_N$, and $\widetilde{H}^k_N$. In particular these cohomologies can be used to distinguish \emph{non-isomorphic} diffeomorphic almost complex manifolds.
\end{cor}

In Example~\ref{ex:non-isom} in the next section we give an explicit example of the use of the $N$-cohomology $H^k_N$ to distinguish \emph{non-isomorphic} non-integrable almost complex structures on particular $4$-manifolds. The $J$-twisted $N$-cohomology is left for future study. From Section~\ref{sec:Jcohom} onwards, we focus our attention exclusively on the $J$-cohomology, as this cohomology has the most interesting applications that the authors were able to find, including in the integrable case.

\subsection{An example of distinguishing almost complex structures} \label{sec:non-isom}

In this section we work through a lengthy example of computing the first $N$-cohomology group $H^1_N$ to distinguish \emph{non-isomorphic} non-integrable almost complex structures of a certain special type. The example illustrates the explicit computability of $H^{\bu}_N$.

\begin{ex} \label{ex:non-isom}
Consider the manifolds of the form $M =  X_1 \times X_2 \times X_3 \times X_4$, where each $X_i$ is either $\R$ or $S^1 = \R / \Z$. We take global ``coordinates'' $x= (x_1, x_2, x_3, x_4)$ on $M$, where $x_i$ is a global ``coordinate'' on $X_i$, understood to be periodic with period $1$ when $X_i = S^1$. At a point $x \in M$, the tangent space $T_x M$ is spanned by the tangent vectors $\{ \partial_{x_1}, \partial_{x_2}, \partial_{x_3}, \partial_{x_4} \}$, which we identify with the standard basis $\{ e_1, e_2, e_3, e_4 \}$ of $\R^4$. Consider on $M$ the family of almost complex structures given by
\begin{equation} \label{eq:non-isomJ}
J = \begin{pmatrix}
  0   &   1   &   p   &   0   \\
 -1   &   0   &   0   &    p   \\
  0   &   0   &   0   &   -1   \\
  0   &   0   &   1   &    0   
\end{pmatrix}
\end{equation}
where $p = p(x_1)$ is a function only of $x_1 \in X_1$, which for simplicity we assume to be real analytic. In particular this means that $p'$ has only isolated zeroes. It is trivial to verify that $J^2 = -I$, and a short computation using~\eqref{eq:Nclassical} yields
\begin{equation} \label{eq:non-isomNzero}
\begin{aligned}
N_{12} & = 0, &  N_{13} & = - p' e_2, & N_{14} & = p' e_1, \\
N_{23} & = - p' e_1, & N_{24} & = - p' e_2, & N_{34} & = - p p' e_2,
\end{aligned}
\end{equation}
where $N_{ij} = N(e_i, e_j)$. In particular such a $J$ is integrable if and only if $p$ is constant. Since $H^k_N = H^k_{\DR}$ in the integrable case, we will henceforth assume that $p'$ is \emph{not identically zero}. For future reference we note that if $\alpha = \sum_{j=1}^4 \alpha_j \dd x_j$ is a $1$-form on $M$, then from~\eqref{eq:non-isomNzero} and~\eqref{eq:alg-derivation1} it follows that
\begin{equation} \label{eq:non-isom1formN}
\begin{aligned}
\iota_N \alpha & = -p ' \alpha_2 \dd x_1 \wedge \dd x_3 + p' \alpha_1 \dd x_1 \wedge \dd x_4 - p' \alpha_1 \dd x_2 \wedge \dd x_3 \\
& \qquad {}- p' \alpha_2 \dd x_2 \wedge \dd x_4 - p p' \alpha_2 \dd x_3 \wedge \dd x_4.   
\end{aligned}
\end{equation}

We will compute $H^1_N$ for these non-integrable almost complex manifolds $(M, J)$, and exhibit several isomorphism classes among them. By definition of $H^k_N$, we have
\begin{equation} \label{eq:non-isomH1def}
H^1_N = \frac{(\ker \dd)^1 \cap (\ker \mathcal{L}_N)^1}{\dd (\ker \mathcal{L}_N)^0}.
\end{equation}
Let us first consider the denominator of~\eqref{eq:non-isomH1def}. Suppose $h = h(x_1, x_2, x_3, x_4)$ lies in $(\ker \mathcal{L}_N)^0$. Then we have
\begin{equation*}
0 = \mathcal{L}_N h = (\iota_N \dd - \dd \iota_N) h = \iota_N (\dd h).
\end{equation*}
By~\eqref{eq:non-isom1formN} and our hypothesis that $p' \not\equiv 0$ we immediately deduce that $h$ is independent of $x_1, x_2$. That is,
\begin{equation} \label{eq:non-isomdenom}
{\dd (\ker \mathcal{L}_N)^0} = \{ \dd h : h = h(x_3, x_4) \}.
\end{equation}

Next we consider the numerator of~\eqref{eq:non-isomH1def}. Suppose that $\alpha = \sum_{j=1}^4 \alpha_j \dd x_j \in (\ker \dd)^1 \cap (\ker \mathcal{L}_N)^1$. Then we have
\begin{equation*}
0 = \mathcal{L}_N \alpha = (\iota_N \dd - \dd \iota_N) \alpha = - \dd (\iota_N \alpha).
\end{equation*}
Thus, taking $\dd$ of equation~\eqref{eq:non-isom1formN}, we obtain a $3$-form on $M$ that must vanish. This gives four independent equations. A short calculation shows that these equations are the following:
\begin{equation} \label{eq:non-isom-ex}
\begin{aligned}
-p' \frac{\partial\alpha_2}{\partial x_2} + p' \frac{\partial\alpha_1}{\partial x_1} + p'' \alpha_1 & = 0, \\
p' \frac{\partial\alpha_1}{\partial x_2} + p' \frac{\partial\alpha_2}{\partial x_1} + p'' \alpha_2 & = 0, \\
p' \frac{\partial\alpha_2}{\partial x_4} + p' \frac{\partial\alpha_1}{\partial x_3} + p p' \frac{\partial\alpha_2}{\partial x_1} + (p p')' \alpha_2 & = 0, \\
p'\frac{\partial\alpha_1}{\partial x_4} - p' \frac{\partial\alpha_2}{\partial x_3} + p p' \frac{\partial\alpha_2}{\partial x_2} & = 0.
\end{aligned}
\end{equation}
Moreover, because $\alpha = \sum \alpha_i dx_i \in \ker \dd$, we have $\frac{\partial\alpha_i}{\partial x_j} = \frac{\partial\alpha_j}{\partial x_i}$ for all $i,j$, so the second equation in~\eqref{eq:non-isom-ex} becomes $2 p' \frac{\partial\alpha_2}{\partial x_1} + p'' \alpha_2 = 0$, which can be solved using standard ODE techniques to give
\begin{equation} \label{eq:alpha2}
\alpha_2 (x_1, x_2, x_3, x_4) = D(x_2, x_3, x_4) |p'(x_1)|^{-\frac{1}{2}}.
\end{equation}
We now consider three separate cases:

{\bf Case 1: $X_1 = S^1$.} Since $p(x_1)$ is periodic, by the mean value theorem there exists some $a \in S^1$ such that $p'(a) = 0$. Hence we must have $D \equiv 0$ for $\alpha_2$ to be defined on all of $M$. Therefore in this case, $\alpha_2 \equiv 0$. Plugging into the first equation in~\eqref{eq:non-isom-ex} and solving gives
\begin{equation*}
\alpha_1 (x_1, x_2, x_3, x_4) = C(x_2, x_3, x_4) p'(x_1)^{-1},
\end{equation*}
which similarly forces $C \equiv 0$ and hence $\alpha_1 \equiv 0$. Then all the equations in~\eqref{eq:non-isom-ex} are trivially satisfied. Moreover from $\frac{\partial\alpha_i}{\partial x_j} = \frac{\partial\alpha_j}{\partial x_i}$ we then deduce that $\alpha_3$ and $\alpha_4$ are both functions of only $x_3, x_4$. We conclude that 
\begin{equation*}
(\ker \dd)^1 \cap (\ker \mathcal{L}_N)^1 = \{ \alpha_3 (x_3, x_4) \dd x_3 + \alpha_4 (x_3, x_4) \dd x_4 : \alpha_3 \dd x_3 +\alpha_4 \dd x_4 \in (\ker \dd)^1 \}.
\end{equation*}
Thus in this case from the above equation and~\eqref{eq:non-isomdenom} we conclude that
\begin{equation} \label{eq:case12}
H^1_{N} (M) = H^1_{\DR}(X_3 \times X_4).
\end{equation}

{\bf Case 2: $X_1 = \R$ and there exists $a \in \R$ such that $p'(a) = 0$.} The exact same reasoning works in this case to deduce that $H^1_{N} (M) = H^1_{\DR} (X_3 \times X_4)$ as in Case One.

{\bf Case 3: $X_1 = \R$ and $p'(x_1) \neq 0$ for all $x_1 \in \R$.} Without loss of generality we can assume $p'(x_1) > 0$ for all $x_1 \in \R$. Recall we obtained~\eqref{eq:alpha2} from the second equation of~\eqref{eq:non-isom-ex}. Plugging this into the first equation of~\eqref{eq:non-isom-ex} gives
\begin{equation} \label{eq:alpha1}
\alpha_1 (x_1, x_2, x_3, x_4) = \frac{1}{p'(x_1)} \left( D_{x_2}(x_2, x_3, x_4) \int^{x_1} p'(s)^{\frac{1}{2}} ds + C(x_2, x_3, x_4) \right).
\end{equation}
The constraint $\frac{\partial\alpha_1}{\partial x_2}=\frac{\partial\alpha_2}{\partial x_1}$ then becomes
\begin{equation} \label{eq:pdouble}
-\frac{D}{2} \frac{p''}{(p')^\frac{1}{2}} = C_{x_2} + D_{x_2 x_2} \int p'^{\frac{1}{2}}.
\end{equation}
Differentiating this with respect to $x_1$ gives
\begin{equation*}
-\frac{D}{2} \left( \frac{p'''}{(p')^\frac{1}{2}} - \frac{(p'')^2}{2 p' (p')^\frac{1}{2}} \right) = D_{x_2x_2} (p')^{\frac{1}{2}}.
\end{equation*}
The above simplifies to
\begin{equation} \label{eq:D}
-2 D p' p''' + D (p'')^2 = 4 D_{x_2 x_2} (p')^2.
\end{equation}
We now consider the possible solutions of~\eqref{eq:D}. If $D \not \equiv 0$, then~\eqref{eq:D} has solution
\begin{equation*}
p' (x_1) = E(x_2, x_3, x_4) \cos^2 \left( \sqrt{\frac{D_{x_2 x_2}}{D}} (F(x_2, x_3, x_4) + x_1) \right).
\end{equation*}
But the above equation contradicts the assumption that $p'(x_1) \neq 0$ for all $x_1$ unless $D_{x_2 x_2} = 0$. So we must have $D_{x_2 x_2} = 0$, and then the above equation says $p'(x_1)$ is constant, so we necessarily have that $p(x_1) = A x_1 + B$ for some $A > 0$, and then~\eqref{eq:pdouble} says that $C_{x_2} \equiv 0$. To summarize, so far in Case 3 we have deduced that if $p$ is not linear, then $D \equiv 0$.

\emph{Subcase 3a: $D \equiv 0$.} Substituting this back into equations~\eqref{eq:alpha2} and~\eqref{eq:alpha1} gives $\alpha_2 = 0$ and $\alpha_1 = \frac{C(x_3, x_4)}{p'(x_1)}$. But then the last two equations in~\eqref{eq:non-isom-ex} say that $\alpha_1$ is independent of $x_3, x_4$ so in fact $C$ is a constant. Now all the equations in~\eqref{eq:non-isom-ex} are satisfied. So the condition that $\mathcal{L}_N \alpha = 0$ has forced $\alpha = \tfrac{1}{p'(x_1)} \dd x_1 + \alpha_3(x_1, x_2, x_3, x_4) \dd x_3 + \alpha_4 (x_1, x_2, x_3, x_4) \dd x_4$. The condition $\dd \alpha = 0$ then forces $\alpha_3$ and $\alpha_4$ to be independent of $x_1, x_2$ so in Case 3a we find that
\begin{equation*}
(\ker \dd)^1 \cap (\ker \mathcal{L}_N)^1 \cong \mathrm{span} \left \{ \frac{1}{p'(x_1)} \right \} \oplus \{ \alpha_3 (x_3, x_4) \dd x_3 + \alpha_4 (x_3, x_4) \dd x_4 : \alpha_3 \dd x_3 +\alpha_4 \dd x_4 \in (\ker \dd)^1 \}.
\end{equation*}
Using~\eqref{eq:non-isomdenom} we finally conclude in Case 3a that
\begin{equation} \label{eq:case3a}
H^1_{N} (M) \cong \mathrm{span} \left \{ \frac{1}{p'(x_1)} \right \} \oplus H^1_{\DR} (X_3 \times X_4).
\end{equation}
 
\emph{Subcase 3b: $p(x_1) = A x_1 + B$ for some $A > 0$.} Substituting this back into equations~\eqref{eq:alpha2} and~\eqref{eq:alpha1} gives $\alpha_2 = \tfrac{1}{A^{\frac{1}{2}}} D$ and $\alpha_1 = \tfrac{1}{A}(A^{\frac{1}{2}} x_1 D_{x_2} +C)$, where $D_{x_2 x_2} = 0$ and $C_{x_2} = 0$. In this case the last two equations of~\eqref{eq:non-isom-ex} become 
\begin{equation*}
\begin{aligned}
\frac{1}{A^{\frac{1}{2}}} D_{x_4} + \frac{1}{A}(A^{\frac{1}{2}} x_1 D_{x_2 x_3} + C_{x_3}) + \frac{A}{A^{\frac{1}{2}}} D & = 0, \\
\frac{1}{A}(A^{\frac{1}{2}} x_1 D_{x_2 x_4} + C_{x_4}) - \frac{1}{A^{\frac{1}{2}}} D_{x_3} + \frac{A x_1 + B}{A^{\frac{1}{2}}} D_{x_2} & = 0,
\end{aligned}
\end{equation*}
which the authors were not able to solve in general. However, in the particular case when $X_2 = S^1$ and $X_3 = X_4 = \R$, we claim that there is an infinite dimensional space of solutions, which descends to an infinite dimensional subspace of $H^1_{N} (M)$. To see this, we note that $D_{x_2 x_2} = $ implies $D_{x_2} = 0$, since $x_2$ is a periodic coordinate. The equations above then reduce to
\begin{equation*}
A^{\frac{1}{2}} D_{x_4} + C_{x_3} + A^{\frac{3}{2}} D = 0, \qquad C_{x_4} - A^{\frac{1}{2}} D_{x_3} = 0.
\end{equation*}
This system has a solution if and only if $\frac{\partial}{\partial x_3} (A^{\frac{1}{2}} D_{x_3}) = \frac{\partial}{\partial x_4} ( - A^{\frac{1}{2}} D_{x_4} - A^{\frac{3}{2}} D )$. This constraint simplifies to $D_{x_3 x_3} + D_{x_4 x_4} + A D_{x_4} = 0$, which has an infinite-dimensional space of solutions for $D$. For each of these, we can then solve for $C$. By~\eqref{eq:alpha2} and~\eqref{eq:non-isomdenom} we therefore conclude that $H^1_{N}(\R \times S^1 \times \R^2)$ is infinite-dimensional if $p(x_1) = A x_1 + B$ for some $A > 0$.
\end{ex}
 
 An immediate corollary of Example~\ref{ex:non-isom} is the following.
\begin{cor} \label{non-isomcor}
Consider the manifold $M = S^1 \times \R^3$. The $N$-cohomology allows us to distinguish \emph{at least five non-isomorphic almost complex structures} on $M$ of the form~\eqref{eq:non-isomJ} with $p$ real analytic.
\end{cor}
\begin{proof}
If $p' \equiv 0$ then $J$ is integrable. If $p' \not \equiv 0$, then $J$ will not be integrable, so it cannot be isomorphic to any integrable $J$, as the vanishing or nonvanishing of the Nijenhuis tensor is preserved by isomorphisms of almost complex structures. We will show that there exist at least four non-isomorphic non-integrable almost complex structures of the form~\eqref{eq:non-isomJ}. If we take $X_1 = S^1$ and $X_2 = X_3 = X_4 = \R$, and choose any non-constant periodic function $p(x_1)$ on $X_1$, then equation~\eqref{eq:case12} tells us that $H^1_N (M) = H_{\DR} (\R^2) = 0$. Alternatively, we can also construct $M$ by taking $X_2 = S^1$ and $X_1 = X_3 = X_4 = \R$. By choosing the non-constant function $p(x_1)$ on $\R$ appropriately, we can arrange either Case 3a or Case 3b of Example~\ref{ex:non-isom}. In Case 3a, equation~\eqref{eq:case3a} gives $H^1_N (M) = \R$, and in Case 3b we get that $H^1_N (M)$ is infinite-dimensional. Finally, if we take $X_3 = S^1$ and $X_1 = X_2 = X_4 = \R$, then we can choose $p$ so that Case 3a is satisfied, and thus by~\eqref{eq:case3a} we get $H^1_N (M) = \R^2$. By Corollary~\ref{cor:isomorphic}, there are thus at least four non-isomorphic non-integrable almost complex structure of the form~\eqref{eq:non-isomJ} on $M = S^1 \times \R^3$.
\end{proof}

\begin{rmk} \label{rmk:non-isom}
By computing the other $N$-cohomology groups, it is expected that many more non-isomorphic non-integrable almost complex structures on $S^1 \times \R^3$ of the form~\eqref{eq:non-isomJ} for $p$ real analytic could be distinguished in this way.
\end{rmk}

\section{The $J$-cohomology} \label{sec:Jcohom}

In this section we study the $J$-cohomology $H^k_J(M)$ in detail, and determine several important results, including its finite-dimensionality (or lack thereof) and its relation (in the integrable case) to the $\del\delbar$-lemma and (in the general case) to the $\dd \mathcal{L}_J$-lemma, which we define. We also compute several explicit examples and discuss variations of our results. From now on, when we say $(M, J)$ is \emph{complex} we mean that $J$ is integrable, that is $N = 0$. Some of our results will be valid only in the complex case, while others will hold for general $J$.

\subsection{Some properties of the $J$-cohomology} \label{sec:Jcohomproperties}

We begin with the following lemma.
\begin{lemma} \label{lemma:JcohomDR}
The following equalities hold:
\begin{equation} \label{eq:JcohomDR}
\begin{aligned}
H^0_J & = H^0_{\DR}, \\
H^1_J & = (\ker \mathcal{L}_J)^1 \cap (\ker \dd)^1.
\end{aligned}
\end{equation}
\end{lemma}
\begin{proof}
We first identify the space $(\ker \mathcal{L}_J)^0$. Let $f \in \Omega^0 (M)$. Since $\mathcal{L}_J = \iota_J \dd - \dd \iota_J$, we have $\mathcal{L}_J f = \iota_J \dd f = J^T \dd f$, because $\iota_J$ vanishes on functions and $\iota_J = J^T$ on $1$-forms. But $J^T$ is invertible, so
\begin{equation} \label{eq:kerLJ0}
(\ker \mathcal{L}_J)^0 = (\ker \dd)^0.
\end{equation}
Thus we have $H^0_J = (\ker \mathcal{L}_J)^0 = (\ker \dd)^0 = H^0_{\DR}$. Moreover, again using~\eqref{eq:kerLJ0} we find that $\dd ((\ker \mathcal{L}_J)^0) = 0$, and hence
\begin{equation*}
H^1_J = \frac{ (\ker \mathcal{L}_J)^1 \cap (\ker \dd)^1 }{\dd ((\ker \mathcal{L}_J)^0)} = (\ker \mathcal{L}_J)^1 \cap (\ker \dd)^1
\end{equation*}
which is what we wanted to show.
\end{proof}

Recall that if $J$ is integrable, then from Remark~\ref{rmk:cohomdefns}(b) we have $\mathcal{L}_J = - \dc = - J^{-1} \dd J$. Since $\dd = \partial + \bar \partial$ in the integrable case, it follows that
\begin{equation} \label{eq:LJ}
\mathcal{L}_J = - \dc = i (\partial - \bar \partial) \qquad \text{ when $J$ is integrable}.
\end{equation}
In particular it follows that
\begin{equation} \label{eq:LJd}
\mathcal{L}_J \dd = i (\partial - \bar \partial)(\partial + \bar \partial) = 2 i \partial \bar \partial \qquad \text{ when $J$ is integrable},
\end{equation}
and that
\begin{equation} \label{eq:HJ1integrable}
(\ker \mathcal{L}_J) \cap (\ker \dd) = (\ker \partial) \cap (\ker \bar \partial) \qquad \text{ when $J$ is integrable}.
\end{equation}

\begin{prop} \label{prop:cptcplxinj}
If $M$ is \emph{compact} and \emph{complex}, then the canonical map $H^1_J \to H^1_{\DR}$ is injective.
\end{prop}
\begin{proof}
We have a short exact sequence of chain complexes 
\begin{equation*}
0 \to (\ker \mathcal{L}_J)^{\bu} \to \Omega^{\bu} \to \frac{\Omega^{\bu}}{(\ker \mathcal{L}_J)^{\bu}} \to 0
\end{equation*}
which induces a long exact sequence in cohomology:
\begin{equation*}
\cdots \to H^0 \left( \frac{\Omega^{\bu}}{(\ker \mathcal{L}_J)^{\bu}} \right) \to H^1_J \to H^1_{\DR} \to \cdots
\end{equation*}
We will show that $H^0 \left( \frac{\Omega^{\bu}}{(\ker \mathcal{L}_J)^{\bu}} \right) = 0$, which will immediately imply the result. Let $[f] \in \Omega^0 / (\ker \mathcal{L}_J)^0$. Then $\dd [f] = [\dd f]$, so $\dd [f] = 0 \in \Omega^1 / (\ker \mathcal{L}_J)^1$ means that $\dd f \in (\ker \mathcal{L}_J)^1$, so $f \in (\ker \mathcal{L}_J \dd)^0$. Thus we have
\begin{equation*}
H^0 \left( \frac{\Omega^{\bu}}{(\ker \mathcal{L}_J)^{\bu}} \right) = \frac{(\ker \mathcal{L}_J \dd)^0}{(\ker \mathcal{L}_J)^0}.
\end{equation*}
We have thus reduced the problem to showing that the subspace inclusion $(\ker \mathcal{L}_J)^0 \subseteq (\ker \mathcal{L}_J \dd)^0$ is actually an equality. Let $f \in (\ker \mathcal{L}_J \dd)^0$. Since $J$ is integrable, by~\eqref{eq:LJd}, we have $\partial \bar \partial f = 0$. We claim that $f$ must be locally constant. This is well-known, but we give the argument for completeness. Let $f = u + i v$. Since $M$ is compact, the function $u$ attains a maximum at some point $x_0 \in M$, with $u(x_0) = t_0 \in \R$. Consider the set $V = u^{-1} (t_0)$. Since $u$ is continuous, $V = u^{-1} (t_0)$ is closed. Let $x \in V$. Let $U$ be the domain of a holomorphic coordinate chart centred at $x$. With respect to this chart, $U$ corresponds to an open neighbourhood of $0$ in $\C^m$ and $\frac{\partial^2 f}{\partial z_i \partial \bar{z_j}} = 0$ for all $i,j$. In particular, we have $\Sigma_{i=1}^m  \frac{\partial^2 f}{\partial z_i \partial \bar{z_i}} = 0$, which says that $f$ is harmonic with respect to the standard Euclidean metric on the domain of the holomorphic chart, hence so is its real part $u$. Since $u(y) \leq u(x)$ for all $y \in U$, by the maximum principle $u$ is constant on $U$. Thus $U \subseteq V$, so $V$ is open. Since $V$ is both open and closed, $V$ is a connected component of $M$. The same argument applies to the imaginary part $v$ of $f$. Thus $f$ is locally constant, so $\dd f = 0$. But then $\mathcal{L}_J f = 0$, and hence $(\ker \mathcal{L}_J \dd)^0 \subseteq (\ker \mathcal{L}_J)^0$, which is what we wanted to show.
\end{proof}

\begin{rmk} \label{rmk:cptessential}
The assumption of compactness in Proposition~\ref{prop:cptcplxinj} is essential, as we show in Example~\ref{ex:C} later in this section.
\end{rmk}

For the remainder of this subsection we consider the particular case of K\"ahler manifolds. Recall that $(M, g, J)$ is K\"ahler if $J$ is integrable and $\dd \omega = 0$, where $\omega(X, Y) = g(JX, Y)$ is the $(1,1)$-form associated to a Riemannian metric $g$ with respect to which $J$ is orthogonal. 
\begin{lemma} \label{lemma:Kahler}
Suppose $(M, g, J)$ is K\"ahler. Let $f \in \Omega^0 (M)$. If $\partial \bar \partial f = 0$, then $\Delta_{\dd} f = 0$, where $\Delta_{\dd}$ is the Hodge Laplacian of the metric $g$.
\end{lemma}
\begin{proof}
This follows from the \emph{K\"ahler identities} and their consequences. A good reference is~\cite{Huybrechts}. Specifically, we use the facts that $[\Lambda, \partial] = i \bar \partial^*$ and $\Delta_{\dd} = 2 \Delta_{\bar \partial}$ on K\"ahler manifolds. Here $\Lambda$ is the adjoint of the wedge product with $\omega$, so in particular it has bidegree $(-1,-1)$. Then we have
\begin{align*}
\Lambda (\partial \bar \partial f) & = ([\Lambda, \partial] + \partial \Lambda) (\bar \partial f) \\
& = i \bar \partial^* \bar \partial f + 0 = i \Delta_{\bar \partial} f = \tfrac{i}{2} \Delta_{\dd} f.
\end{align*}
Thus $\partial \bar \partial f = 0$ implies $\Delta_{\dd} f = 0$.
\end{proof}

Later, in Corollary~\ref{cor:mainthmKahler}, we will show that the $J$-cohomology of a \emph{compact K\"ahler} manifold is isomorphic to its de Rham cohomology. But this will use the deep fact that K\"ahler manifolds satisfy the $\partial \bar \partial$-lemma. The next result is a fairly direct proof of the isomorphism $H^k_J \cong H^k_{\DR}$ for compact K\"ahler manifolds when $k = 0, 1, 2m$, using Hodge theory. The authors include it here as it is instructive.

\begin{prop} \label{prop:isomKahler1}
Let $(M, g, J)$ be a \emph{compact K\"ahler} manifold. Then $H^k_J (M) \cong H^k_{\DR} (M)$ for $k = 0, 1, 2m$.
\end{prop}
\begin{proof}
The case $k=0$ is always true by Lemma~\ref{lemma:JcohomDR}. Consider the case $k=1$. By Lemma~\ref{lemma:JcohomDR}, we have $H^1_J = (\ker \mathcal{L}_J)^1 \cap (\ker \dd)^1$. Let $\alpha$ be a closed $1$-form. By Hodge theory, we can write $\alpha = \alpha_{\mathcal H} + \dd f$ for some $\Delta_{\dd}$-harmonic $1$-form $\alpha_{\mathcal H}$ and some smooth function $f$. Suppose that $\alpha \in \ker \mathcal{L}_J$ as well. Then $\mathcal{L}_J \alpha_{\mathcal H} + \mathcal{L}_J \dd f = 0$. Since $\alpha_{\mathcal H}$ is harmonic, it is $\dd$-closed, and hence also $\partial$-closed and $\bar\partial$-closed. In particular by~\eqref{eq:LJ} we have $\mathcal{L}_J \alpha_{\mathcal H} = 0$. Then~\eqref{eq:LJd} and Lemma~\ref{lemma:Kahler} give $\Delta_{\dd} f = 0$, so $f$ is locally constant and thus $\dd f = 0$. We have thus shown that $H^1_J$ is precisely the space of harmonic $1$-forms on $M$, which by Hodge theory is isomorphic to $H^1_{\DR}$.

Finally, consider the case $k = 2m$. By Hodge theory, the space $H^{2m}_{\DR}$ is isomorphic to the space of harmonic $2m$-forms, which are the forms of the type $f \omega^m$ for $f$ a harmonic function. We have $H^{2m}_J = \Omega^{2m} (M) / \dd (\ker \mathcal{L}_J)^{2m-1}$. Let $\alpha \in \Omega^{2m} (M)$. Then $\alpha = h \omega^m$ for some smooth function $h$ on $M$. By Hodge theory, we can write $h = f + \Delta_{\dd} \phi$ for some harmonic function $f$ and some smooth function $\phi$. Thus, to establish that $H^{2m}_J \cong H^{2m}_{\DR}$, it suffices to show that $(\Delta_{\dd} \phi) \omega^m$ always lies in the image of $\dd \mathcal{L}_J$, because by~\eqref{eq:LJ} we have $(\mathcal{L}_J)^2 = 0$ in the integrable case, so a $(2m-1)$-form in the image of $\mathcal{L}_J$ will necessarily be in the kernel of $\mathcal{L}_J$. We observe by~\eqref{eq:LJd} that $\dd \mathcal{L}_J \phi = - \mathcal{L}_J \dd \phi = - 2 i \partial \bar \partial \phi$. Now the $(1,1)$-form $\dd \mathcal{L}_J \phi = - 2 i \partial \bar \partial \phi$ can be written in the form $\sigma \omega + \gamma$, for some function $\sigma$ and some \emph{primitive} $(1,1)$-form $\gamma$. This means that $\Lambda \gamma = 0$ and $\gamma \wedge \omega^{m-1} = 0$. We also have $\Lambda \omega = m$. (See~\cite{Huybrechts}, for example.) Taking $\Lambda$ of both sides and using the proof of Lemma~\ref{lemma:Kahler}, we find that
\begin{equation*}
\Lambda (\dd \mathcal{L}_J \phi) = \Lambda( - 2 i \partial \bar \partial \phi ) = \Delta_{\dd} \phi = \Lambda (\sigma \omega + \gamma) = m \sigma.
\end{equation*}
Thus we have $\Delta_{\dd} \phi = m \sigma$, and therefore we have
\begin{equation*}
(\dd \mathcal{L}_J \phi) \wedge \omega^{m-1} = (\sigma \omega + \gamma) \wedge \omega^{m-1} = \sigma \omega^m = \tfrac{1}{m} (\Delta_{\dd} \phi) \omega^m.
\end{equation*}
But then using the fact that $\mathcal{L}_J \omega = 0$ and $\dd \omega = 0$ in the K\"ahler case, and the fact that both $\mathcal{L}_J$ and $\dd$ are derivations, we find
\begin{equation*}
\tfrac{1}{m} (\Delta_{\dd} \phi) \wedge \omega^m = (\dd \mathcal{L}_J \phi) \wedge \omega^{m-1} = \dd \mathcal{L}_J (\phi \omega^{m-1}),
\end{equation*}
which is what we wanted to show.
\end{proof}

\begin{rmk} \label{rmk:Riemannsurface}
Proposition~\ref{prop:isomKahler1} is enough to conclude that the $J$-cohomology is isomorphic to the de Rham cohomology for any compact Riemann surface. However, as we mentioned above, we will show later that this isomorphism holds for any compact K\"ahler manifold and in fact such an isomorphism holds for any almost complex manifold (not necessarily compact) satisfying the $\dd \mathcal{L}_J$-lemma, which reduces to the $\partial \bar \partial$-lemma in the integrable case. This is all discussed in Section~\ref{sec:del-delbar}.
\end{rmk}

The next example shows that both Propositions~\ref{prop:isomKahler1} and~\ref{prop:cptcplxinj} fail in the non-compact case.
\begin{ex} \label{ex:C}
Let $M$ be a K\"ahler manifold with $b^1 = \dim H^1_{\DR} = 0$. We will compute $H^1_J$. By Lemma~\ref{lemma:JcohomDR}, we have $H^1_J = (\ker \mathcal{L}_J)^1 \cap (\ker \dd)^1$. Let $\alpha \in (\ker \dd)^1$. Since $b^1 = 0$, we have $\alpha = \dd f$ for some function $f$. Then $\mathcal{L}_J \alpha = \mathcal{L}_J \dd f = 2 i \partial \bar \partial f = 0$ by~\eqref{eq:LJd} and thus since $M$ is K\"ahler, $f$ is harmonic by Lemma~\ref{lemma:Kahler}. Thus we have shown that
\begin{equation*}
H^1_J (M) = \{ \dd f : f \text{ is a harmonic function on $M$} \}.
\end{equation*}
In particular, if $M = \C$ then $H^1_J$ is \emph{infinite-dimensional}.
\end{ex}

\begin{rmk} \label{rmk:Cexample}
Example~\ref{ex:C} shows that $H^1_J (\C) \ncong H^1_J (\text{point})$, and therefore the $J$-cohomology \emph{does not} satisfy homotopy invariance. Moreover, there is no Mayer-Vietoris sequence, Poincar\'e duality, or K\"unneth formula for $H^k_J$ in general. However, $H^k_J$ does have functoriality, as we showed in Proposition~\ref{prop:functoriality}.
\end{rmk}

\subsection{The dimension of $H_J^{\bu} (M)$ for compact $M$} \label{sec:finitedim}

In this section, we study the $J$-cohomology of a \emph{compact} almost complex manifold $(M,J)$. We will prove that when $M$ is a compact \emph{complex} manifold, the $J$-cohomology groups are all \emph{finite-dimensional}. In the general (non-integrable) almost complex case, we only have a partial result, which says that the spaces $H^k_J (M)$ are finite-dimensional for $k = 0, 1, 2m$. The main ingredient for the proof of both of these results is Hodge theory.

First we consider the complex case. We will need to make use of the \emph{Bott-Chern Laplacian}. We state only the facts about this operator and its corresponding Hodge decomposition that we will need. More details can be found, for example, in~\cite[Section 2b]{Schweitzer}.

\begin{thm}[\cite{Schweitzer}] \label{thm:BC-Hodge}
Let $M$ be a compact complex manifold. Equip $M$ with a Riemannian metric $g$ compatible with the complex structure $J$. There is a degree-preserving fourth order elliptic linear differential operator $\Delta_{\BC} : \Omega^{\bullet} (M) \otimes \C \to \Omega^{\bullet} (M) \otimes \C$, called the \emph{Bott-Chern Laplacian}, on the space of complex-valued forms. For each $k = 0, 1 , \ldots, 2m$, it induces a Hodge-type decomposition on $\Omega^k(M) \otimes \C$ as follows:
\begin{equation*}
\Omega^k \otimes \C = (\ker \Delta_{\BC})^k \underset{\perp}{\oplus} (\im \partial \bar{\partial})^k \underset{\perp}{\oplus} (\im \partial^* + \im \bar{\partial}^*)^k
\end{equation*}
where $\partial^*$ and $\bar \partial^*$ denote the formal adjoints with respect to $g$. This decomposition is orthogonal with respect to $g$. Moreover, $(\ker \Delta_{\BC})^k$ is finite-dimensional and we have
\begin{equation} \label{eq:BC2}
\Delta_{\BC} \alpha = 0 \iff \partial \alpha =\bar{\partial} \alpha = \partial^* \bar{\partial}^* \alpha = 0
\end{equation}
and
\begin{equation} \label{eq:BC3}
\partial \alpha = \bar{\partial} \alpha=0 \iff \alpha \in (\ker \Delta_{\BC})^k \underset{\perp}{\oplus} (\im \partial \bar{\partial})^k.
\end{equation}
\end{thm}

We now use Theorem~\ref{thm:BC-Hodge} to prove the finite-dimensionality of the $J$-cohomology in the compact complex case.

\begin{thm} \label{thm:fd-complex}
Let $(M,J)$ be a compact complex manifold. Then $H^k_J (M)$ is a finite-dimensional vector space for all $k$.
\end{thm}
\begin{proof}
Since $\dd = \partial + \bar \partial$, by~\eqref{eq:BC2} any $\beta \in (\ker \Delta_{\BC})^k$ lies in $( (\ker \dd) \otimes \C)^k$, and thus represents an element of $H^k_J \otimes \C$. We claim that this natural map $\ker(\Delta_{\BC})^k \to H^k_J \otimes \C$ is surjective. Assuming this claim, by the finite-dimensionality of $(\ker \Delta_{\BC})^k$ it follows that $H^k_J \otimes \C$, and hence $H^k_J$, is finite-dimensional as required.

We now prove the claim. Let $\alpha \in ( (\ker \dd) \otimes \C)^k \cap ( (\ker \mathcal{L}_J) \otimes \C)^k$. Then from~\eqref{eq:HJ1integrable} we in fact have $\partial \alpha = \bar{\partial} \alpha = 0$. By~\eqref{eq:BC3}, we can write $\alpha = \beta + \partial \bar{\partial} \gamma$, for some $\beta, \gamma$ with $\Delta_{\BC} \beta = 0$.  Note that $\partial \bar{\partial} \gamma = \dd (\bar{\partial} \gamma - \tfrac{1}{2} \dd \gamma)$. But using $\partial^2 = \bar{\partial}^2 = 0$ and $\partial \bar \partial = - \bar \partial \partial$, we also have
\begin{align*}
\mathcal{L}_J (\bar{\partial} \gamma - \tfrac{1}{2} \dd \gamma) & = i (\partial - \bar{\partial}) (\bar{\partial} \gamma - \tfrac{1}{2} \dd \gamma) \\
& = i (\partial - \bar\partial)(\bar \partial \gamma - \tfrac{1}{2} \partial \gamma - \tfrac{1}{2} \bar\partial \gamma) \\
& = i ( \partial \bar \partial \gamma - \tfrac{1}{2} \partial \bar \partial \gamma + \tfrac{1}{2} \bar \partial \partial \gamma) = 0.
\end{align*}
Hence, $\alpha - \beta \in ( \dd (\ker \mathcal{L}_J) \otimes \C )^k$, and thus we have $[\alpha] = [\beta]$ as equivalence classes in $H_J^k \otimes \C$, which establishes the claim.
\end{proof}

Next we consider the general case, when $J$ need not be integrable. In this case we can only obtain partial results.

\begin{thm} \label{thm:fd-general}
Let $(M,J)$ be a compact almost complex manifold of real dimension $n=2m$. Then $H^k_J (M)$ is a finite-dimensional vector space for $k = 0,1, n$.
\end{thm}

We have already seen in~\eqref{eq:JcohomDR} that $H^0_J (M) \cong H_{\DR}^0 (M)$, so the key point of Theorem~\ref{thm:fd-general} concerns $H^1_J$ and $H^{2m}_J$. Before we can prove Theorem~\ref{thm:fd-general}, we need to establish two lemmas.

Recall from Section~\ref{sec:deriv} that $\dd = \mathcal{L}_I$, where $I$ is the identity endomorphism of $TM$. Consider the differential operator $D = \dd +i \mathcal{L}_J = \mathcal{L}_{I + i J}$ on the bundle of complexified differential forms $\Omega^{\bu} \otimes \C$. First we observe from $\mathcal{L}_J \dd = - \dd \mathcal{L}_J$ and~\eqref{eq:LJsquared} that $D^2 = (\dd + i \mathcal{L}_J)(\dd + \i \mathcal{L}_J) = - (\mathcal{L}_J)^2 = \mathcal{L}_N$ so $D^2 = 0$ if and only if $N = 0$ if and only if $J$ is integrable. If $J$ is integrable, then~\eqref{eq:LJ} implies that $D = 2\bar{\partial}$, and it is a standard fact that the $\bar \partial$-Laplacian $\Delta_{\bar{\partial}} = \bar \partial \bar \partial^* + \bar \partial^* \bar \partial$ is an elliptic operator.

In the general not necessarily integrable case, we claim that the symbol of the $D$-Laplacian $\Delta_D = D D^* + D^* D$ is again elliptic, and in fact equals the symbol of $4 \Delta_{\bar \partial}$. This can be computed directly. Alternatively, as mentioned in Remark~\ref{rmk:cohomdefns}(b) above, in~\cite{dKS} it is shown that $\mathcal{L}_J = -J^{-1} \dd J - \iota_{J \cdot N}$, and one can compute using the formulas established in~\cite[Section 3]{dKS} that
\begin{equation*}
D = \dd + i \mathcal{L}_J = 2 \bar \partial - \tfrac{1}{4} \iota_N - \tfrac{3}{4} i \iota_{J \cdot N}.
\end{equation*}
Thus $D = 2 \bar \partial$ up to lower order terms, which establishes the claim. This observation, together with the usual Hodge decomposition for elliptic operators on compact manifolds, establishes the following lemma.

\begin{lemma} \label{lemma:fd1}
The operator $\Delta_D = D D^* + D^* D$ is elliptic, so in particular, the space $\ker \Delta_D = (\ker D) \cap (\ker D^*)$ is finite-dimensional, and
\begin{equation*}
\Omega^k \otimes \C = (\ker \Delta_D)^k \oplus \big( (\im D)^k + (\im D^*)^k \big).
\end{equation*}
The sum $(\im D)^k + (\im D^*)^k$ is not direct in general because $D^2 \neq 0$.
\end{lemma}

We will now employ the tool of spectral sequences. A good reference is~\cite{BT}. Consider the following double complex:
\begin{equation*} 
\newcommand{\ra}[1]{\xrightarrow{\quad#1\quad}}
\newcommand{\da}[1]{\left\downarrow{\scriptstyle#1}\vphantom{\displaystyle\int_0^1}\right.}
\begin{array}{llllllllllll}
              &                          &    \Omega^0     \\
              &                          &    \da{\dd}       \\       
  \Omega^0    &    \ra{\mathcal{L}_J}   &    \Omega^1     \\
  \da{\dd}      &                          &    \da{\dd}       \\       
  \Omega^1    &    \ra{\mathcal{L}_J}   &    \Omega^2     \\  
  \vdots      &    \ddots                &    \vdots       \\  
  \Omega^{k-2}&    \ra{\mathcal{L}_J}   &    \Omega^{k-1} \\
  \da{\dd}      &                          &    \da{\dd}       \\
  \Omega^{k-1}&    \ra{\mathcal{L}_J}   &    \Omega^k     \\
  \da{\dd}      &                          &    \da{\dd}       \\
  \Omega^k    &    \ra{\mathcal{L}_J}   &    \Omega^{k+1} \\
  \vdots      &    \ddots                &    \vdots       \\
  \Omega^{n-1}&    \ra{\mathcal{L}_J}   &    \Omega^n     \\
  \da{\dd}      &                          &                 \\
  \Omega^n    &                          &                 \\
\end{array}
\end{equation*}
If we calculate the cohomology in the downward direction, we obtain the de Rham cohomology groups, which are finite-dimensional. Hence, the total complex of this double complex has finite-dimensional cohomology. We can also calculate the cohomology to the right and then downwards. In the second page of the spectral sequence, we obtain:
\begin{equation*}
\begin{tikzcd}
 \vdots                  &  &    \vdots                      \\  
 H^{k-2}_J & &  H^{k-1}((\coker \mathcal{L}_J)^{\bu}) \arrow[lldd, "\delta^{k-1} = \dd \mathcal{L}_J^{-1} \dd"near end] \\
 H^{k-1}_J & &  H^k((\coker \mathcal{L}_J)^{\bu})     \\
 H^k_J  & &   H^{k+1}((\coker \mathcal{L}_J)^{\bu})  \\
  \vdots                &   &    \vdots                     
\end{tikzcd}
\end{equation*}
Moreover, the spectral sequence stabilizes after the second page. Consider the group $H^k_J$. Because the total complex has finite cohomology, $H^k_J$ will become finite-dimensional when we quotient it out by the image of $\delta^{k-1}$. We have thus established the following lemma.

\begin{lemma} \label{lemma:fd2}
The space $H^k_J$  is finite-dimensional if and only if
\begin{equation} \label{eq:fd}
\im \delta^{k-1} = \frac{\{ \dd v \in \Omega^k : \mathcal{L}_J v = \dd u \text{ for some } u \} + \dd (\ker  \mathcal{L}_J)^{k-1}}{\dd (\ker \mathcal{L}_J)^{k-1}}
\end{equation}
is finite-dimensional.
\end{lemma}

We can now prove Theorem~\ref{thm:fd-general}. For brevity, we will suppress ${} \otimes \C$ throughout the proof.
\begin{proof}[Proof of Theorem~\ref{thm:fd-general}]
We need to show that~\eqref{eq:fd} is finite-dimensional for $k=1, 2m$. Consider first the case $k=1$. We showed in the proof of Lemma~\ref{lemma:JcohomDR} that $\dd (\ker \mathcal{L}_J)^0 = 0$, so by~\eqref{eq:fd} $\im \delta^0 = \{ \dd v \in \Omega^k : \mathcal{L}_J v = \dd u \text{ for some } u \}$ in this case. Let $\dd v \in \im \delta^0$ with $\mathcal{L}_J v = \dd u$. Define $f = u + i v$. Using~\eqref{eq:LJspecial} and $(\iota_J)^2 = -I$ on $1$-forms, we have
\begin{align*}
D f & = (\dd + i \mathcal{L}_J) (u + iv) = (\dd u - \iota_J \dd v) + i (\dd v + \iota_J \dd u) \\
& = (\dd u -\iota_J \dd v) + i \iota_J (\dd u - \iota_J \dd v) = 0.
\end{align*}
Since $\Delta_D = D^*D$ on $\Omega^0$, we find that $f$ lies in the finite-dimensional space $\ker \Delta_D$, and thus $v = \imag f$ lies in a finite-dimensional space.

Next we consider the case when $k=n=2m$. This time the space $\im \delta^{n-1}$ in~\eqref{eq:fd} does not simplify. But elements of the quotient space $\im \delta^{n-1}$ are represented by exact $n$-forms. Let $[\dd \tau] \in \im \delta ^{n-1}$, and define $\sigma = -\iota_J \tau$. Set $\mu = \sigma + i \tau$. Using~\eqref{eq:LJspecial} and $(\iota_J)^2 = -I$ on $(n-1)$-forms, we have
\begin{align*}
D \mu & = (\dd + i \mathcal{L}_J) (\sigma + i \tau) = (\dd \sigma + \dd \iota_J \tau) + i (\dd \tau - \dd \iota_J \sigma) \\
& = \dd (\sigma + \iota_J \tau) + i \dd \iota_J( - \iota_J \tau - \sigma)= 0,
\end{align*}
and hence $\mu \in \ker D$. From the Hodge decomposition of Lemma~\ref{lemma:fd1}, we can write $\mu = \alpha + D \beta$, where $\alpha \in \ker \Delta_D$ and $D^2 \beta = 0$. Now we observe that $0 = D^2 \beta = -\mathcal{L}_J^2 \beta$, so $\mathcal{L}_J \beta \in \ker \mathcal{L}_J$. Moreover, we have $\dd D \beta = i \dd \mathcal{L}_J \beta$, so in fact $\dd D\beta \in \dd(\ker \mathcal{L}_J)^{n-1}$. Therefore, $\dd \mu = \dd \alpha$ modulo $\dd(\ker \mathcal{L}_J)^{n-1}$, and hence $[\dd \tau]$ lies in the finite-dimensional vector space that is the imaginary part of $\dd ( \ker \Delta_D)$.
\end{proof}

\subsection{The relation between the $J$-cohomology and the $\partial \bar{\partial}$ lemma} \label{sec:del-delbar}

In this section we relate the $J$-cohomology to the $\partial \bar{\partial}$-lemma~\cite{Angella, Huybrechts} in the integrable case, and to a generalization of the $\partial \bar{\partial}$-lemma which we call the $\dd \mathcal{L}_J$-lemma, in the general case. Along the way we will discuss the connections to previously known results. We begin with the statement of the classical $\partial \bar{\partial}$-lemma.

\begin{defn}[$\partial \bar{\partial}$-lemma] \label{defn:ddbar}
Let $(M,J)$ be a complex manifold. If for every $\alpha \in \Omega^k (M) \otimes \C$, we have
\begin{equation*}
(\text{$\alpha$ is $\partial$-closed and $\bar\partial$-closed and $\dd$-exact}) \Longleftrightarrow (\text{$\alpha$ is $\partial \bar \partial$-exact})
\end{equation*}
then we say that $(M, J)$ satisfies the $\partial \bar \partial$-lemma in degree $k$.
\end{defn}

Recall that the inclusion $(\ker \mathcal{L}_J)^{\bu} \hookrightarrow \Omega^{\bu}$ induces a well-defined natural map $H_J^{\bu} \to H_{\DR}^{\bu}$.

\begin{thm} \label{thm:main}
Let $M$ be a \emph{complex} manifold. Then $M$ satisfies the $\partial \bar{\partial}$-lemma \textcolor{red}{for all degrees $k$} if and only if the natural map $H_J^k (M) \to H_{\DR}^k (M)$ is an isomorphism \textcolor{red}{for all degrees $k$}.
\end{thm}

Theorem~\ref{thm:main} is not new. A version (adapted to a different setting) can be found, for example, in~\cite[Theorem 4.2]{Cavalcanti-ddc}. For our purposes, Theorem~\ref{thm:main} is just a special case of the more general Theorem~\ref{thm:main2} that we prove later in this section.

\begin{rmk} \label{rmk:BCA}
Another well-known characterization of the $\partial \bar \partial$-lemma is in terms of the Bott-Chern and Aeppli cohomologies. These are cohomologies on complex manifolds defined using the second order differential operator $\partial \bar \partial$. These cohomologies are isomorphic to the de Rham cohomology if and only if the $\partial \bar \partial$-lemma holds. See~\cite[Theorem 1.14]{Angella} for a thorough discussion. As far as the authors are aware, these characterizations require the assumption of compactness. By contrast, Theorem~\ref{thm:main} does not require $M$ to be compact.
\end{rmk}

Since any \emph{compact K\"ahler} manifold satisfies the $\partial \bar{\partial}$-lemma in all degrees~\cite{Huybrechts}, we obtain the following corollary of Theorem~\ref{thm:main}, which is the promised generalization of Proposition~\ref{prop:isomKahler1}.

\begin{cor} \label{cor:mainthmKahler}
If $M$ is a \emph{compact K\"ahler} manifold, then the natural map $H_J^{k} (M) \to H_{\DR}^{k} (M)$ is an isomorphism for all $k$.
\end{cor}

Before we can state and prove the more general version of Theorem~\ref{thm:main}, namely Theorem~\ref{thm:main2}, we need to reformulate and then generalize the $\partial \bar\partial$-lemma to the not necessarily integrable case. Recall from~\eqref{eq:LJ} and~\eqref{eq:LJd} that in the integrable case, we have $2 i \partial \bar \partial = \mathcal{L}_J \dd = - \dc \dd = \dd  \dc$, where we have used from~\eqref{eq:commdL} that $[ \dd, \mathcal{L}_J ] = \dd \mathcal{L}_J + \mathcal{L}_J \dd = 0$. Moreover, is is clear that $(\ker \dd) \cap (\ker \dc) = (\ker \partial) \cap (\ker \bar \partial)$. It follows that the $\partial \bar \partial$-lemma is equivalent to the following $\dd \dc$-lemma~\cite{Angella}.

\begin{defn}[$\dd \dc$-lemma] \label{defn:ddc}
Let $(M,J)$ be a complex manifold. If for every $\alpha \in \Omega^k (M) \otimes \C$, we have
\begin{equation*}
(\text{$\alpha$ is $\dc$-closed and $\dd$-exact}) \Longleftrightarrow (\text{$\alpha$ is $\dd \dc$-exact})
\end{equation*}
then we say that $(M, J)$ satisfies the $\dd \dc$-lemma in degree $k$.

Equivalently, $(M, J)$ is said to satisfy the $\dd \dc$-lemma in degree $k$ if and only if
\begin{equation} \label{eq:ddc}
\frac{(\im \dd)^k \cap (\ker \mathcal{L}_J)^k}{(\im \dd \mathcal{L}_J)^k} = 0
\end{equation}
since $\dc = - \mathcal{L}_J$ in the integrable case.
\end{defn}

\begin{lemma} \label{lemma:ddcprelim}
Let $M$ be a complex manifold. After complexification, we have an isomorphism
\begin{equation} \label{eq:Pisom}
\frac{(\im \dd)^k \cap (\ker \mathcal{L}_J)^k}{(\im \dd \mathcal{L}_J)^k} \cong \frac{(\im \mathcal{L}_J)^k \cap (\ker \dd)^k}{(\im \dd \mathcal{L}_J)^k}.
\end{equation}
\end{lemma}
\begin{proof}
Define a linear map $P : \Omega^{\bu} \otimes \C \to \Omega^{\bu} \otimes \C$ by $P = (-1)^p \pi^{p,q}$, where $\pi^{p,q}$ is the projection of $\Omega^{\bu} \otimes \C$ onto $\Omega^{p,q}$. Let $\alpha \in \Omega^{p,q}$. Then using~\eqref{eq:LJ} we compute
\begin{align*}
- i P \dd \alpha & = - i P (\partial \alpha + \bar \partial \alpha) = - i \big( (-1)^{p+1} \partial \alpha + (-1)^p \bar \partial \alpha \big) \\
& = (-1)^p ( i \partial \alpha - i \bar \partial \alpha ) = \mathcal{L}_J P \alpha.
\end{align*}
By linearity, we deduce that $\mathcal{L}_J P = - i P \dd$ on all forms. We also have $P^2 = I$. Thus $P = P^{-1}$ is an isomorphism that 
interchanges $(\im \dd)^k$ with $(\im \mathcal{L}_J)^k$ and interchanges $(\ker \mathcal{L}_J)^k$ with $(\ker \dd)^k$. The result follows.
\end{proof}

\begin{defn}[$\dd \mathcal{L}_J$-lemma] \label{defn:dLJ}
Let $(M, J)$ be an \emph{almost complex} manifold. We say that $(M, J)$ satisfies the $\dd \mathcal{L}_J$-lemma in degree $k$ if
\begin{equation} \label{eq:dLJ}
\frac{(\im \mathcal{L}_J)^k \cap (\ker \dd)^k}{(\im \dd \mathcal{L}_J)^k} = 0.
\end{equation}

By Lemma~\ref{lemma:ddcprelim} and equation~\eqref{eq:ddc}, in the case when $J$ is \emph{integrable}, the $\dd \mathcal{L}_J$-lemma is equivalent to the $\dd \dc$-lemma.
\end{defn}

We now state and prove a more general version of Theorem~\ref{thm:main}.

\begin{thm} \label{thm:main2}
Let $(M,J)$ be an almost complex manifold, and denote by $\phi^{k} : H_J^{k} \to H_{\DR}^{k}$ the natural induced map. Then $M$ satisfies the $\dd \mathcal{L}_J$-lemma in degree $k$ if and only if $\phi^k$ is injective and $\phi^{k-1}$ is surjective.
\end{thm}
\begin{proof}
The short exact sequence of chain complexes
\begin{equation*}
0 \to (\ker \mathcal{L}_J)^{\bu} \xrightarrow{\phi^{\bu}} \Omega^{\bu} \xrightarrow{\mathcal{L}_J} (\im \mathcal{L}_J)^{\bu \textcolor{red}{+1}} \to 0
\end{equation*}
induces a long exact sequence 
\begin{equation*}
\cdots \to H^{k-1}_J \xrightarrow {\phi^{k-1}} H^{k-1}_{\DR} \to H^{\textcolor{red}{k}} \big( (\im \mathcal{L}_J)^{\bu} \big) \to H^k_J \xrightarrow {\phi^k} H^k_{\DR} \to \cdots
\end{equation*}
in cohomology. By exactness we deduce that
\begin{equation*}
(\phi^{k-1} \text{ is surjective and } \phi^k \text{ is injective}) \iff H^{\textcolor{red}{k}} \big( (\im \mathcal{L}_J)^{\bu} \big) = 0.
\end{equation*}
Thus by Definition~\ref{defn:dLJ}, the proof would be complete if we can show that
\begin{equation} \label{eq:main2temp}
H^{\textcolor{red}{k}} \big( (\im \mathcal{L}_J)^{\bu} \big) = \frac{(\ker \dd)^k \cap (\im \mathcal{L}_J)^k}{(\im \dd \mathcal{L}_J)^k}.
\end{equation}
But this follows directly from the definition, because
\begin{equation*}
H^{\textcolor{red}{k}} \big( (\im \mathcal{L}_J)^{\bu} \big) = \frac{ (\ker \dd)^k \cap (\im \mathcal{L}_J )^k }{\dd ( (\im \mathcal{L}_J)^{k-1})} = \frac{ (\ker \dd)^k \cap (\im \mathcal{L}_J )^k }{(\im \dd \mathcal{L}_J)^k}.
\end{equation*}
\end{proof}

\subsection{Examples of the applications of Theorems~\ref{thm:main} and~\ref{thm:main2}} \label{sec:mainex}

In this section we discuss in detail three explicit examples where we apply Theorem~\ref{thm:main} and the more general Theorem~\ref{thm:main2}. The examples serve both to illustrate that the $J$-cohomology can in some cases be explicitly computed, and that it can be used to establish the non-validity of the $\partial \bar \partial$-lemma in the integrable case or more generally the non-validity of the $\dd \mathcal{L}_J$-lemma in the non-integrable case.

The first two examples we consider are both integrable: the Hopf manifolds and the Iwasawa manifold. It is well-known~\cite{Angella, Angella2, Huybrechts} that these complex manifolds do not satisfy the $\partial \bar \partial$-lemma. The existing proofs of these facts that the authors were aware of in the literature either involve
\begin{itemize}
\item the computation of the Bott-Chern or Aeppli cohomologies~\cite{Angella} (see Remark~\ref{rmk:BCA}), or
\item the explicit demonstration that the Massey products are nontrivial~\cite{Huybrechts} (these products must vanish for any compact complex manifold for which the $\partial \bar \partial$-lemma holds.)
\end{itemize}
By contrast, our demonstration of the failure of the $\partial \bar \partial$-lemma to hold for the Hopf manifolds and the Iwasawa manifold in both cases is by establishing that $H^1_J \not \equiv H^1_{\DR}$ and is very direct and straightforward.

At the end of this section we also consider an explicit example of a \emph{non-integrable} almost complex manifold $(T^4, J)$ that does not satisfy the $\dd \mathcal{L}_J$-lemma, which is again established by the explicit computation of $H^1_J$.

\begin{ex} \label{ex:Hopf}
Consider the quotient $P$ of $\C^2 \setminus \{ 0 \}$ by the $\Z$ action generated by $z \mapsto 2 z$. The space $P$ is a \emph{compact complex manifold}, called the Hopf surface. We will show that $H^1_J (P) = 0$.

Suppose $\alpha = f_1 \dd z_1 + f_2 \dd z_2 + \bar{f}_1 \dd \bar{z}_1 + \bar{f}_2 \dd \bar{z}_2 \in H^1_J (P)$. By~\eqref{eq:JcohomDR} and~\eqref{eq:HJ1integrable} we have $\partial \alpha = \bar \partial \alpha = 0$. We can also consider $\alpha$ as a $1$-form on $\C^2 \setminus \{ 0 \}$ that is invariant under the action $\lambda(z) = 2 z$. That is, $\lambda^* \alpha = \alpha$, which is equivalent to the equations
\begin{equation} \label{eq:Hopfinv}
2 f_1(2z) = f_1(z), \qquad 2 f_2(2z) = f_2(z).
\end{equation}
The fact that $\bar{\partial} \alpha = 0$ says that $f_1$ and $f_2$ are holomorphic functions on $\C^2 \setminus \{ 0 \}$. Suppose that $f_i (z) \neq 0$ for some $z$. Then by~\eqref{eq:Hopfinv} we have
\begin{equation*}
|f(0)| = \lim_{n \to \infty} \left| f \left( \frac{z}{2^n} \right) \right| = \lim_{n \to \infty} 2^n |f(z)| = \infty.
\end{equation*}
But this contradicts Hartog's extension theorem, which says that we can extend each $f_i$ to a holomorphic function on the whole $\C^2$. Thus we must have $f_i \equiv 0$. Therefore $\alpha = 0$ and hence $H^1_J (P) = 0$.
\end{ex}

\begin{rmk} \label{rmk:Hopf}
A similar argument performed along eigenvectors shows that $H^1_J (P) = 0$ holds for a general Hopf manifold $P$ obtained by quotienting $\C^n \setminus \{ 0 \}$ by the $\Z$-action generated by a linear dilation $\lambda(z) = Az$. Since $H^1_{J} (P) = 0$ for all these, and $P \cong S^1 \times S^{2n-1}$ has $H^1_{\DR} (P) = \R$, we deduce from  Theorem~\ref{thm:main} that none of the Hopf manifolds satisfy the $\partial \bar{\partial}$-lemma.
\end{rmk}

\begin{ex} \label{ex:Iwasawa}
The Iwasawa manifold $W$ is defined to be the quotient of the set
\begin{equation*}
\left \{ \begin{pmatrix} 1 & z_1 & z_2 \\ 0 & 1 & z_3 \\ 0 & 0 & 1 \end{pmatrix} : z_1, z_2, z_3 \in \C \right \} \cong \C^3
\end{equation*}
by the left action of the group
\begin{equation*}
\left \{ \begin{pmatrix} 1 & a & b \\ 0 & 1 & c \\ 0 & 0 & 1 \end{pmatrix} : a, b, c \in \Z [i] \right \}.
\end{equation*}
We will compute $H^1_J (W)$. Suppose $\alpha = f_1 \dd z_1 + f_2 \dd z_2 + f_3 \dd z_3 + \bar{f}_1 \dd \bar{z}_1 + \bar{f}_2 \dd \bar{z}_2 + \bar{f}_3 \dd \bar{z}_3 \in H^1_J (W)$. As in Example~\ref{ex:Hopf}, consider $\alpha$ as a $1$-form on $\C^3$ that is invariant under the covering transformations. From the equation
\begin{equation*}
\begin{pmatrix} 1 & a & b \\ 0 & 1 & c \\ 0 & 0 & 1 \end{pmatrix} \begin{pmatrix} 1 & z_1 & z_2 \\ 0 & 1 & z_3 \\ 0 & 0 & 1 \end{pmatrix} = \begin{pmatrix} 1 & z_1 + a & z_2 + b + a z_3 \\ 0 & 1 & z_3 + c \\ 0 & 0 & 1 \end{pmatrix}
\end{equation*}
we find that the invariance of $\alpha$ under the covering transformations is equivalent to the system of equations
\begin{equation} \label{eq:Iwasawa}
\begin{aligned}
f_1(z_1 + a, z_2 + a z_3, z_3) & = f_1(z_1, z_2, z_3), \\
f_2(z_1 + a, z_2 + a z_3, z_3) & = f_2(z_1, z_2, z_3), \\
f_2(z_1 + a, z_2 + a z_3, z_3) + f_3(z_1 + a, z_2 + a z_3, z_3) & = f_3(z_1, z_2, z_3), \\
f_i(z_1, z_2 + b, z_3) & = f_i(z_1, z_2, z_3), \quad \text{for } i=1, 2, 3, \\
f_i(z_1, z_2, z_3 + c) & = f_i(z_1, z_2, z_3), \quad \text{for } i=1, 2, 3,
\end{aligned}
\end{equation}
for all $a,b,c \in \mathbb{Z}[i]$. Also as in Example~\ref{ex:Hopf}, by~\eqref{eq:JcohomDR} and~\eqref{eq:HJ1integrable} we deduce that the $f_i$ are holomorphic functions on $\C^3$. The last two equations in~\eqref{eq:Iwasawa} state that the $f_i$ are doubly periodic, with periods $1$ and $i$, in the variables $z_2$, $z_3$, so by Liouville's Theorem each $f_i$ must be constant in those two variables. By an abuse of notation, we write $f_i (z_1, z_2, z_3) = f_i(z_1)$, and the system~\eqref{eq:Iwasawa} becomes
\begin{equation} \label{eq:Iwasawab}
\begin{aligned}
f_1(z_1 + a) & = f_1(z_1), \\
f_2(z_1 + a) & = f_2(z_1), \\
f_2(z_1 + a) + f_3(z_1 + a) & = f_3(z_1).
\end{aligned}
\end{equation}
The first two equations in~\eqref{eq:Iwasawab} say that $f_1$, $f_2$ are doubly periodic in $z_1$ as well, so they are in fact constant functions. If we write $f_2 = c$, then the last equation in~\eqref{eq:Iwasawab} becomes $c + f_3(z_1 + a) = f_3(z_1)$ for all $a \in \Z[i]$, which implies that $c = f_2 =0$ and thus that $f_3$ is also constant.

We have therefore shown that $H^1_J (W) = \mathrm{span} \{ \dd x_1, \dd x_3 \}$, because recall that we only deal with real forms. But $\dim H^1_{\DR}(W) = 4$, as can be seen, for example, in~\cite[Appendix A]{Angella2}). Hence $H^1_J (W) \not \equiv H^1_{\DR} (W)$ and the Iwasawa manifold does not satisfy the $\partial \bar{\partial}$-lemma by Theorem~\ref{thm:main}.
\end{ex}

\begin{ex} \label{ex:T4}
Consider the $4$-dimensional torus $T^4 = \R^4 / \Z^4$ with the following \emph{nonstandard} almost complex structure
\begin{equation*}
J = \begin{pmatrix} 0 & 1 & f & -g \\ -1 & 0 & g & f \\ 0 & 0 & 0 & -1 \\ 0 & 0 & 1 & 0 \end{pmatrix}
\end{equation*}
where $f(x_1, x_2, x_3,x_4)$ and $g(x_1, x_2, x_3, x_4)$ are functions which are periodic in each coordinate with period 1. It is trivial to verify that $J^2 = -I$. Define
\begin{equation} \label{eq:AB}
A = f_{x_2} + g_{x_1}, \qquad B = f_{x_1} - g_{x_2}. 
\end{equation}
A computation using~\eqref{eq:Nclassical} yields
\begin{equation} \label{eq:T4Nzero}
\begin{aligned}
& N_{12} = 0, \qquad N_{13} = N_{24} = A \partial_{x_1} - B \partial_{x_2}, \qquad N_{14} = - N_{23} = B \partial_{x_1} + A \partial_{x_2}, \\
& N_{34} = (fA - gB) \partial_{x_1} - (fB + gA) \partial_{x_2},
\end{aligned}
\end{equation}
where $N_{ij} = N(e_i, e_j)$. In particular, we deduce that this $J$ is integrable if and only if $A = B = 0$. But these conditions clearly imply that $f_{x_1 x_1} + f_{x_2 x_2} = g_{x_1 x_1} + g_{x_2 x_2} = 0$, so $f$ and $g$ must be harmonic in $x_1$, $x_2$, and by compactness $f$, $g$ must be constant in $x_1$, $x_2$. That is, we conclude that
\begin{equation} \label{eq:T4integrable}
\text{this $J$ is integrable if and only if $f$, $g$ are constant in $x_1$, $x_2$.}
\end{equation}

Now let $\alpha \in H^1_J (T^4) = (\ker \mathcal{L}_J)^1 \cap (\ker \dd)^1$. Then $\dd \alpha = 0$, so since $H^1_{\DR} (\R^4) = 0$, there exists a $u \in C^{\infty} (\R^4)$ such that $\alpha = \dd u$. Note that $u$ need not descend to a function on $T^4$. That is, it need not be $\Z$-periodic in each coordinate, and indeed this will be the case if and only if $[\alpha] = 0$ in $H^1_{\DR} (T^4)$. Since $\alpha \in \ker \mathcal{L}_J$, we have $\mathcal{L_J} \dd u = - \dd \iota_J \dd u = 0$. Thus $\iota_J \dd u$ is a closed $1$-form, so there exists some $v \in C^{\infty} (\R^4)$ such that $\iota_J \dd u = -\dd v$, and again $v$ need not descent to $T^4$. Since $\iota_J = J^T$ on $1$-forms, and $(J^T)^2 = - I$, we can rewrite this equation as $\dd u = J^T \dd v$. Explicitly, we have
\begin{equation} \label{eq:T4}
\begin{pmatrix} u_{x_1} \\ u_{x_2} \\ u_{x_3} \\ u_{x_4} \end{pmatrix} = \begin{pmatrix} 0 & -1 & 0 & 0 \\ 1 & 0 & 0 & 0 \\ f & g & 0 & 1 \\ -g & f & -1 & 0 \end{pmatrix} \begin{pmatrix} v_{x_1} \\ v_{x_2} \\ v_{x_3} \\ v_{x_4} \end{pmatrix} = \begin{pmatrix} -v_{x_2} \\ v_{x_1} \\ f v_{x_1} + g v_{x_2} + v_{x_4} \\ -g v_{x_1} + f v_{x_2} - v_{x_3} \end{pmatrix}.
\end{equation}

We claim that the natural map $H^1_J (T^4) \to H^1_{\DR} (T^4) = \R^4$ is injective. Let $[\alpha] \in H^1_J (T^4)$. If $[\alpha] = 0$ in $H^1_{\DR} (T^4)$, then $\alpha = \dd u$ for some $u \in C^{\infty} (T^4)$. That is, the function $u$ into~\eqref{eq:T4} is $\Z$-periodic in all four variables. The first two equations in~\eqref{eq:T4} show that $u_{x_1 x_1} + u_{x_2 x_2} = 0$, so $u$ is harmonic in the first two coordinates. Therefore, if $u$ must be constant in $x_1$ and $x_2$, and in particular $u_{x_1} = u_{x_2} = v_{x_1} = v_{x_2} = 0$. After substituting this into the last two equations, we find that $u$ is also harmonic in $x_3$ and $x_4$, which proves that $u$ is constant and $\dd u = 0$. Therefore, $\alpha = 0$, and the map $H^1_J (T^4) \to H^1_{\DR} (T^4) = \R^4$ is indeed injective, and hence
\begin{equation} \label{eq:T4inj}
\dim H^1_J (T^4) \leq \dim H^1_{\DR} (T^4) = 4.
\end{equation}
(Note that we cannot just apply Proposition~\ref{prop:cptcplxinj} to derive~\eqref{eq:T4inj} because $J$ need not be integrable.)

Now we solve~\eqref{eq:T4}, without the assumption that $u$ descends to $T^4$. Note that the first two equations in~\eqref{eq:T4} imply that $u_{x_1}$ is harmonic in the first two coordinates, so $u_{x_1} = a(x_3,x_4)$ is constant in $x_1$ and $x_2$. We will show that $a$ is in fact a constant. Note that for any $x_2, x_3, x_4$, we have
\begin{equation*}
a(x_3, x_4) = \int_0^1 a(x_3, a_4) \dd t = \int_0^1 u_{x_1}(t, x_2, x_3, x_4) \dd t = u(1, x_2, x_3, x_4) - u(0, x_2, x_3, x_4).
\end{equation*}
Differentiating the above expression with respect to $x_3$ gives
\begin{equation*}
a_{x_3} (x_3, x_4) = u_{x_3}(1, x_2, x_3, x_4) - u_{x_3}(0, x_2, x_3, x_4).
\end{equation*}
Since $u_{x_3}$ is a coefficient of the $1$-form $\alpha = \dd u$ on $T^4$, it is $\Z$-periodic, and thus the right hand side above vanishes. Hence $a$ is independent of $x_3$ and similarly independent of $x_4$. So indeed we conclude that $u_{x_1} = a$ is a constant function.

Similarly, $u_{x_2}$, $v_{x_1}$, and $v_{x_2}$ are all constant functions. By equality of mixed partials we now deduce that $u_{x_3}$, $u_{x_4}$, $v_{x_3}$, $v_{x_4}$ are all constant in $x_1$, $x_2$. Thus by~\eqref{eq:T4} we find that both $f v_{x_1} + g v_{x_2}$ and $-g v_{x_1} + f v_{x_2}$ must be constant in $x_1$, $x_2$. That is,
\begin{align*}
(f v_{x_1} + g v_{x_2})_{x_1} & = f_{x_1} v_{x_1} + g_{x_1} v_{x_2} = 0, \\
(f v_{x_1} + g v_{x_2})_{x_2} & =  f_{x_2} v_{x_1} + g_{x_2} v_{x_2} = 0, \\
(-g v_{x_1} + f v_{x_2})_{x_1} & = - g_{x_1}  v_{x_1} + f_{x_1} v_{x_2} = 0, \\
(-g v_{x_1} + f v_{x_2})_{x_2} & = - g_{x_2} v_{x_1} + f_{x_2} v_{x_2} = 0, \\
\end{align*}
which can also be written in the following matrix form:
\begin{equation*}
\begin{pmatrix} f_{x_1} & f_{x_2} \\ g_{x_1} & g_{x_2} \\ \end{pmatrix} \begin{pmatrix} v_{x_2} & -v_{x_1} \\ v_{x_1} & v_{x_2} \end{pmatrix} = \begin{pmatrix} 0 & 0 \\ 0 & 0 \end{pmatrix}.
\end{equation*}
Since the determinant of the second factor above is $(v_{x_2})^2 + (v_{x_1})^2$, we conclude that either $v_{x_1} = v_{x_2} = 0$, or else the second factor is invertible and thus the first matrix must be zero, which means that $f$, $g$ are constant in $x_1$, $x_2$.

\emph{Case One:} If $f$, $g$ are not constant in $x_1$, $x_2$, then $v_{x_1} = v_{x_2} = 0$, and in this case the solution set of~\eqref{eq:T4} is
\begin{equation*}
\{ (u_{x_1}, u_{x_2}, u_{x_3}, u_{x_4}) = (0, 0, s, t) : s, t \in \R \}.
\end{equation*}

\emph{Case Two:} If $f,g$ are constant in $x_1$, $x_2$, we claim that~\eqref{eq:T4} is solvable for any values of the two constants $u_{x_1}$, $u_{x_2}$. In fact, for any $u_{x_1}$, $u_{x_2}$, we put $v_{x_1} = u_{x_2}$, $v_{x_2} = -u_{x_1}$ and then solve
\begin{equation} \label{eq:T4b}
u_{x_3 x_3} + u_{x_4 x_4} = F,
\end{equation}
where
\begin{equation} \label{eq:T4F}
F = v_{x_1} f_{x_3} + v_{x_2} f_{x_4} + v_{x_2} g_{x_3} - v_{x_1} g_{x_4}.
\end{equation}
This is because we can substitute any solution $(u_{x_3}, u_{x_4})$ of~\eqref{eq:T4b} into~\eqref{eq:T4} and then solve for $v_{x_3}$, $v_{x_4}$, as we do not require $v$ to descend to $T^4$. From~\eqref{eq:T4F} and the $\Z$-periodicity of $f$, $g$, we see that $\int_{T^4} F \, \dd x_1 \wedge \dd x_2 \wedge \dd x_3 \wedge \dd x_4 = 0$, so by Hodge theory we can write $F = \Delta \phi$, but then again from~\eqref{eq:T4F} and the constantcy of $f$, $g$, in $x_1$, $x_2$, we find that
\begin{equation*}
\Delta \phi_{x_1} = F_{x_1} = 0, \qquad \Delta \phi_{x_2} = F_{x_2} = 0,
\end{equation*} 
so in fact both $\phi_{x_1}$ and $\phi_{x_2}$ are constant functions, and $(u_{x_3}, u_{x_4}) = (\phi_{x_3}, \phi_{x_4})$ is a solution of~\eqref{eq:T4b}. Together with the solutions $(u_{x_1}, u_{x_2}, u_{x_3}, u_{x_4}) = (0, 0, s, t)$, we find that the solution space is at least $4$-dimensional in this case, hence exactly $4$-dimensional by~\eqref{eq:T4inj}.

In summary, we have concluded that
\begin{equation*}
\dim H^1_J (T^4) = \begin{cases} 4 & \text{if $f$, $g$ are constant in $x_1$, $x_2$ }, \\ 2 & \text{otherwise}. \end{cases}
\end{equation*}
By~\eqref{eq:T4integrable} the first case occurs if and only if $J$ is integrable, which agrees with Theorem~\ref{thm:main}. Note also by Theorem~\ref{thm:main2} that $(T^4, J)$ \emph{does not} satisfy the $\dd \mathcal{L}_J$-lemma in the nonintegrable case, because $H^1_J \not \equiv H^1_{\DR}$.
\end{ex}

\subsection{Variations of Theorem~\ref{thm:main}}

In this section, we will consider some variations of Theorem~\ref{thm:main}.

Let $(M, J)$ be a complex manifold. We claim there is a natural linear map $\psi^k : H^k_J \to H^k_{\bar \partial}$, where $H^k_{\bar \partial} = (\ker \bar \partial)^k / (\im \bar \partial)^k$. For $[\alpha] \in H^k_J$, any representative $\alpha$ is both $\dd$-closed and $\mathcal{L}_J$-closed, and thus by~\eqref{eq:HJ1integrable} is $\bar \partial$-closed. We define $\psi^k [\alpha]$ to be the class of $\alpha$ in $H^k_{\bar \partial}$. We need to check that this is well-defined. If $[\alpha] = [\alpha'] \in H^k_J$, then $\alpha = \alpha' + \dd \beta$ for some $\beta \in \ker \mathcal{L}_J$. By~\eqref{eq:LJ} we have $\partial \beta = \bar \partial \beta$, so $\dd \beta = 2 \bar \partial \beta$, and hence $[\alpha] = [\alpha'] \in H^k_{\bar \partial}$.

\begin{thm} \label{thm:maincplx}
Let $M$ be a complex manifold, and let $\psi^{\bu}: H_J^{\bu} \to H_{\bar \partial}^{\bu}$ be the natural map. Then $M$ satisfies the $\partial \bar{\partial}$-lemma in degree $k$ if and only if $\psi^k$ is injective and $\psi^{k-1}$ is surjective.
\end{thm}
\begin{proof}
It follows from~\eqref{eq:LJ} that when restricted to $(\ker \mathcal{L}_J)^{\bu}$, the maps $\dd$ and $\bar \partial$ are equal modulo a factor, so they give rise to the same cohomology. Hence the proof of Theorem~\ref{thm:maincplx} is identical to that of Theorem~\ref{thm:main2}.
\end{proof}

Similarly, we can prove the following corollary in the same way that we prove Theorem~\ref{thm:main}.

\begin{cor} \label{cor:mainpsi}
The complex manifold $M$ satisfies the $\partial \bar{\partial}$-lemma (in all degrees) if and only if $\psi^k$ is an isomorphism for all $k$.
\end{cor}

Combining Corollary~\ref{cor:mainpsi} and Theorem~\ref{thm:main}, we obtain a proof of the following result.

\begin{cor} \label{cor:maincong}
If the complex manifold $M$ satisfies the $\partial \bar{\partial}$-lemma in all degrees, then $H^k_{\DR} (M) \cong H^k_{\bar \partial} (M)$ for all $k$.
\end{cor}

In fact, from our proof of Theorem~\ref{thm:main}, we in fact get the following more general result.

\begin{thm} \label{thm:main3}
Let $C^{\bu}$ be a chain complex with differential $\dd$, and $L : C^{\bu} \to C^{\bu}[s]$ be a chain map for some integer $s$, where $C^{\bu}[s]$ denotes the $s$-shifted chain complex of $C^{\bu}$. Then the following two statements are equivalent:
\begin{enumerate}[i)] \setlength\itemsep{-1mm}
\item For every $\alpha \in C^{\bu}$, we have $\alpha \in (\ker \dd) \cap  (\im L) \iff \alpha \in (\im \dd L)$,
\item The natural map $H ((\ker L)^{\bu}) \to H(C^{\bu})$ is an isomorphism.
\end{enumerate}
\end{thm}

For example, we can take $C^{\bu}$ to be $\Omega^{\bu}_{\mathrm{cs}}$, the forms with compact support, or $\Omega^{\bu}_{\mathrm{r}}$, the forms which vanish to order $r$ at infinity, and $L$ to be any Nijenhuis-Lie derivation on differential forms.

\section{Future directions} \label{sec:future}

A number of natural directions for future study are immediate. These include the following:
\begin{itemize}
\item Can one establish finite-dimensionality of $H^k_J$ for all $k$ when $(M, J)$ is compact but \emph{non-integrable}? In Theorem~\ref{thm:fd-general} we proved this for $k=0,1,2m$. It is reasonable to believe that it holds more generally. Similarly, for compact $M$, the $N$-cohomology $H^{\bu}_N$ and the twisted $N$-cohomology $\widetilde{H}^{\bu}_N$ may well be finite-dimensional.
\item Are there any more applications of $H^{\bu}_N$, and what applications of $\widetilde{H}^{
\bu}_N$ can be found? More generally, is there a geometric meaning to the $N$-cohomology and the twisted $N$-cohomology? Recall that these cohomologies \emph{only} make sense for non-integrable almost complex structures.
\item A very natural and important class of \emph{non-integrable} almost complex manifolds are the \emph{nearly K\"ahler} manifolds. See~\cite{CaS}  for a survey. These manifolds are equipped with a compatible Hermitian metric, but the associated $(1,1)$-form $\omega$ is \emph{not} closed. It is reasonable that the applications of and geometric meanings of $H^{\bu}_J$, $H^{\bu}_N$, and $\widetilde{H}^{\bu}_N$ may be more rich for compact nearly K\"ahler manifolds.
\item Similar cohomology theories can be studied in the context of manifolds with $\G$-structure. See~\cite{CKT} and the references therein for more details.
\end{itemize}

\end{document}